\documentclass{amsart}

\usepackage[utf8]{inputenc}
\newcommand{\href}[1]{#1} % does nothing, but defines the command so the
\newcommand{\texorpdfstring}[2]{#1} % does nothing, but defines the command

\usepackage{ifthen}
\usepackage[top=4.5cm,bottom=4.5cm,left=3.5cm,right=3.5cm]{geometry}
\usepackage[all]{xy}
\xyoption{arc}

\usepackage{amsmath,amssymb,amstext,latexsym,amsthm,amscd} % Lots of math symbols and environments

\newtheorem{thm}{Theorem}[section]
\newtheorem*{thm*}{Theorem}
\newtheorem{lemma}[thm]{Lemma}
\newtheorem*{lemma*}{Lemma}
\newtheorem{prop}[thm]{Proposition}
\newtheorem*{prop*}{Proposition}
\newtheorem{cor}[thm]{Corollary}

\newtheorem*{fact}{Fact}

\theoremstyle{definition}
\newtheorem{df}[thm]{Definition}

\newtheorem{ex}[thm]{Example}

\newtheorem{rmk}[thm]{Remark}

\theoremstyle{remark}

\def\ol{\overline}

\newcommand{\bbB}{\mathbb{B}}
\newcommand{\bbQ}{\mathbb{Q}}
\newcommand{\bbL}{\mathbb{L}}
\newcommand{\bbF}{\mathbb{F}}
\newcommand{\bbC}{\mathbb{C}}
\newcommand{\bbZ}{\mathbb{Z}}
\newcommand{\bbP}{\mathbb{P}}
\newcommand{\bbA}{\mathbb{A}}

\newcommand{\bbH}{\mathbb{H}}

\newcommand{\bbT}{\mathbb{T}}

\renewcommand{\AA}{\bbA}
\newcommand{\QQ}{\bbQ}

\newcommand{\BB}{\bbB}
\newcommand{\CC}{\bbC}
\newcommand{\FF}{\bbF}

\newcommand{\ZZ}{\bbZ}

\newcommand{\PP}{\bbP}
\newcommand{\TT}{\bbT}

\newcommand{\Fbar}{\overline{\FF}}

\newcommand{\cA}{\mathcal{A}}
\newcommand{\cB}{\mathcal{B}}
\newcommand{\cC}{\mathcal{C}}
\newcommand{\cD}{\mathcal{D}}
\newcommand{\cO}{\mathcal{O}}
\newcommand{\cF}{\mathcal{F}}

\newcommand{\cK}{\mathcal{K}}
\newcommand{\cL}{\mathcal{L}}
\newcommand{\cH}{\mathcal{H}}
\newcommand{\cR}{\mathcal{R}}
\newcommand{\cM}{\mathcal{M}}

\newcommand{\cE}{\mathcal{E}}
\newcommand{\cP}{\mathcal{P}}
\newcommand{\cT}{\mathcal{T}}
\newcommand{\cU}{\mathcal{U}}
\newcommand{\cV}{\mathcal{V}}

\DeclareMathOperator{\Mot}{\bf Mot }

\DeclareMathOperator{\Sch}{\bf Sch }

\newcommand{\fE}{\mathfrak E}
\newcommand{\fG}{\mathfrak G}
\newcommand{\fV}{\mathfrak V}

\newcommand{\fF}{\mathfrak F}
\newcommand{\fm}{\mathfrak m}

\newcommand{\fS}{\mathfrak S}

\newcommand{\fX}{\mathfrak X}
\newcommand{\fY}{\mathfrak Y}

\DeclareMathOperator{\Gal}{Gal}
\DeclareMathOperator{\Spec}{Spec}

\DeclareMathOperator{\Spf}{Spf}

\DeclareMathOperator{\Hom}{Hom}
\DeclareMathOperator{\End}{End}

\DeclareMathOperator{\Aut}{Aut}
\DeclareMathOperator{\car}{char}
\DeclareMathOperator{\gl}{GL}
\DeclareMathOperator{\GL}{GL}
\DeclareMathOperator{\PGL}{PGL}
\DeclareMathOperator{\tr}{tr}

\DeclareMathOperator{\Pic}{Pic}

\DeclareMathOperator{\res}{res}

\DeclareMathOperator{\img}{img}

\DeclareMathOperator{\Sym}{Sym}
\DeclareMathOperator{\ev}{ev}

\DeclareMathOperator{\chow}{CH} %Be careful!
\DeclareMathOperator{\cl}{cl}
\DeclareMathOperator{\AJ}{AJ}
\DeclareMathOperator{\Ext}{Ext}
\DeclareMathOperator{\nrd}{nrd}
\DeclareMathOperator{\Corr}{Corr}
\DeclareMathOperator{\pr}{pr}
\DeclareMathOperator{\SL}{SL}
\DeclareMathOperator{\PSL}{PSL}

\DeclareMathOperator{\Fil}{Fil}
\DeclareMathOperator{\red}{red}
\DeclareMathOperator{\MF}{MF}
\DeclareMathOperator{\Rep}{Rep}

\DeclareMathOperator{\Dst}{D_{\st}}
\DeclareMathOperator{\Vst}{V_{\st}}

\DeclareMathOperator{\stab}{stab}
\DeclareMathOperator{\emb}{emb}
\DeclareMathOperator{\Ind}{Ind}
\DeclareMathOperator{\Maps}{Maps}

\newcommand{\dfn}{{:=}}

\newcommand{\injects}{\hookrightarrow}
\newcommand{\id}{\ensuremath \text{Id}}
\newcommand{\tns}[1][]{\otimes_{#1}}
\newcommand{\mtx}[4]{\left(\begin{matrix}#1&#2\\#3&#4\end{matrix}\right)}

\newcommand{\smat}[1]{\left(\begin{smallmatrix}#1\end{smallmatrix}\right)}
\newcommand{\smtx}[4]{\left(\begin{smallmatrix}#1&#2\\#3&#4\end{smallmatrix}\right)}

\newcommand{\If}{\ensuremath{\text{if }}}
\newcommand{\Else}{\ensuremath{\text{else}}}

\newcommand{\emphh}[2][ ]{%
\ifthenelse{\equal{#1}{ }}{\index{default}{#2}{\emph{#2}}}{\index{default}{#1@#2}{\emph{#2}}}%
}

\newcommand{\emphhh}[3][ ]{%
\ifthenelse{\equal{#1}{ }}{\index{default}{#2}{\emph{#3}}}{\index{default}{#1@#2}{\emph{#3}}}%
}

\def\sumprime{\mathop{\sum{\raise3pt\hbox{${}'$}}}} %Copied from somewhere else
\def\prodprime{\mathop{\prod{\raise3pt\hbox{${}'$}}}}

\newcommand{\et}{{\text{et}}}
\newcommand{\dR}{{\text{dR}}}

\newcommand{\cris}{{\text{cris}}}
\newcommand{\Het}{H_\et}
\newcommand{\Hcris}{H_\cris}
\newcommand{\Hdr}{H_\dR}
\newcommand{\Hdrc}{H_{\text{dR,c}}}
\newcommand{\cHdr}{\cH_\dR}

\newcommand{\Hst}{H_\st}
\newcommand{\Char}{C_{\text{har}}}

\newcommand{\Log}{{\operatorname{Log}}}

\newcommand{\ad}{{\text{ad}}}

\renewcommand{\epsilon}{\varepsilon}
\newcommand{\mmu}{\mu\!\!\!\mu}
\newcommand{\Rmax}{{\cR^{\text{max}}}}
\newcommand{\an}{\text{an}}
\newcommand{\ur}{\text{ur}}

\newcommand{\st}{\text{st}}
\newcommand{\Bst}{B_{\st}}

\newcommand{\tto}[1]{%
\ifthenelse{\equal{#1}{}}{\to}{\stackrel{#1}{\longrightarrow}}}

\newcommand{\abs}[1]{|{#1}|}

\newcommand{\stt}{\,\mid\,}

\newcommand{\fixme}[1]{}

\begin{document}
\title{CM cycles on Shimura curves,\\and $p$-adic $L$-functions}
\author{Marc Masdeu}
\email{masdeu@math.columbia.edu}
\address{Columbia University\\
Rm 415, MC 4441\\
2990 Broadway\\
New York , NY 10027}

%\classification{11G40 (11F11 11G18)}
%\keywords{CM cycles, Shimura Curve, $p$-adic integration, anti-cyclotomic $p$-adic $L$-function}
\thanks{}

\renewcommand{\emphh}[1]{{\emph{#1}{}}}
\begin{abstract}
Let $f$ be a modular form of weight $k\geq 2$ and level $N$, let $K$ be a quadratic imaginary field, and assume that there is a prime $p$ exactly dividing $N$. Under certain arithmetic conditions on the level and the field $K$, one can attach to this data a $p$-adic $L$-function $L_p(f,K,s)$, as done by Bertolini-Darmon-Iovita-Spie\ss{} in~\cite{bertolini2002tse}. In the case of $p$ being inert in $K$, this analytic function of a $p$-adic variable $s$ vanishes in the critical range $s=1,\ldots,k-1$, and therefore one is interested in the values of its derivative in this range. We construct, for $k\geq 4$, a Chow motive endowed with a distinguished  collection of algebraic cycles which encode these values, via the $p$-adic Abel-Jacobi map.

Our main result generalizes the result obtained by Iovita-Spie\ss{} in~\cite{iovita2003dpa}, which gives a similar formula for the central value $s=k/2$. Even in this case our construction is different from the one found in~\cite{iovita2003dpa}.
\end{abstract}
\maketitle

\tableofcontents
\section*{Acknowledgements}
The author thanks Henri Darmon and Adrian Iovita for their generosity and the help and support received during this project. The author also appreciates the valuable comments made by Bas Edixhoven and the anonymous referee.
\newpage
\section{Introduction}

Fix a quadratic imaginary field $K$, and let $f$ be a modular form defined over $\QQ$. The goal of the different theories of $p$-adic $L$-functions is to produce rigid-analytic functions attached to $f$ that interpolate the Rankin-Selberg $L$-function $L(f/K,s)$ in different ways. The theory has so far developed in two directions, which correspond to the two independent $\ZZ_p$-extensions of the field $K$: the cyclotomic and anti-cyclotomic extension.

The first approach to such $p$-adic analogues appeared in~\cite{MR0354674}, where Mazur and Swinnerton-Dyer constructed a
$p$-adic $L$-function associated to a modular form $f$ of arbitrary even
weight $n+2$ using the cyclotomic $\ZZ_p$-extension of $\QQ$. Mazur,
Tate and Teitelbaum formulated in~\cite{MR830037} a conjectural formula that related the order of vanishing of this $p$-adic $L$ function to that $L(f,s)$, and in~\cite{MR1198816}, Greenberg and Stevens proved that formula in the case of weight $2$. Also,
Perrin-Riou~\cite{MR1168369} obtained a Gross-Zagier type formula for the central value using $p$-adic
heights.

Nekov\'{a}\v{r}~\cite{MR1343644} extended the result of Perrin-Riou to higher weights, by using the definition of $p$-adic height that he had already introduced in his earlier paper~\cite{MR1263527}. Combining this result with his previous work on Euler systems~\cite{MR1135466}, he obtained a result of Kolyvagin-type for the cyclotomic $p$-adic $L$-function.

Bertolini and Darmon, in a series of papers~\cite{MR1419003},~\cite{MR1614543} and~\cite{MR1695201}, constructed another $p$-adic $L$-function
which depends instead on the \emph{anti-cyclotomic} $\ZZ_p$-extension of a
fixed quadratic imaginary field $K$. One important feature of this
construction is that it is purely $p$-adic, unlike its cyclotomic
counterpart. Bertolini and Darmon formulated
the analogous conjectures to those of
Teitelbaum~\cite{teitelbaum1990voa}, and proved them in the case of weight $2$.

After these papers, Iovita and Spie\ss{} entered the project initiated by
Bertolini and Darmon. Their goal was to generalize the previous constructions and results to
higher weights and to make them more conceptual. They
considered certain Selmer groups, in the spirit of the conjectures of
Bloch and Kato~\cite{bloch1990fat} as generalized by Fontaine and
Perrin-Riou in~\cite{MR1265546}.

Assume that the level $N$ of $f$ can be factored as $pN^-N^+$, where $N^-$ is the product of an odd number of distinct primes none of which equals $p$, and let $K$ be a quadratic imaginary number field such that:
\begin{enumerate}
\item All prime divisors of $N^-$ are inert in $K$, and
\item all prime divisors of $N^+$ are split in $K$.
\end{enumerate}
In~\cite{bertolini2002tse}, taking ideas from the work of Schneider in~\cite{MR756097}, the four
authors construct under these restrictions the anti-cyclotomic $p$-adic $L$-function attached to the rigid modular form $f$ and the quadratic imaginary field $K$, and obtain a formula which computes the derivative of this $p$-adic
$L$-function at the central point in terms of an integral on the
$p$-adic upper half plane, using the integration theory introduced by
Coleman in~\cite{coleman1985tpc}. Assume for simplicity that the ideal class number of $K$ is $1$.  Using their techniques one can easily show that, when $p$ is inert in $K$, the anti-cyclotomic $p$-adic $L$-function vanishes at all the critical values. Moreover, one computes a formula for the derivative at all the
 values in the critical range: if $f$ is a modular form of even weight $n+2$, and we denote by $L_p(f,K,s)$ the anti-cyclotomic $p$-adic $L$-function attached to $f$ and $K$, then for all $0\leq j\leq n$, one has:
\begin{equation}
\label{eq:general}
 L_p'(f,K,j+1)=\int_{\ol z_0}^{z_0} f(z)(z-z_0)^{j}(z-\ol z_0)^{n-j}
dz,
\end{equation}
where, $z_0,\ol z_0\in \cH_p(K)$ are certain conjugate Heegner points on
the $p$-adic upper-half plane. In $2003$, Iovita and Spie\ss{}~\cite{iovita2003dpa}
interpreted the quantity appearing in the right hand side of formula~\eqref{eq:general}, in the case of $j=\frac{n}{2}$, as the image of a Heegner cycle under a $p$-adic analogue to the Abel-Jacobi map. This paper gives a similar geometric interpretation of the quantity
appearing in the right hand side of the previous formula, for all values of $j$.

Let $M_{n+2}(X)$ denote the space of modular forms on a Shimura curve $X$, of weight
$n+2\geq 4$. The case of weight $2$ is excluded for technical reasons, and because it has already been studied by other authors. Let $K$ be a quadratic imaginary field in which $p$ is inert, and fix an elliptic curve $E$  with
complex multiplication. In this setting, we construct a Chow motive $\cD_{n}$ over $X$, and a collection
of algebraic cycles $\Delta_\varphi$ supported in the fibers over CM-points of $X$,
 indexed by  isogenies
 $\varphi\colon E\to E'$, of elliptic
curves with complex multiplication. The motive $\cD_{n}$ is obtained from a self-product of a certain number of abelian surfaces, together with a self-product of the elliptic curve $E$. The cycles $\Delta_\varphi$ are essentially the graph of $\varphi$, and are expected to carry more information than the classical Heegner cycles.

One can define a map analogous to the classical Abel-Jacobi map for curves, but for varieties defined over $p$-adic fields. This map, denoted $\AJ_{K,p}$, assigns to a null-homologous algebraic cycle an element in the dual of the de Rham realization of the motive $\cD_{n}$. This motive has precisely been constructed so that this realization is
\[
M_{n+2}(X)\tns[\QQ_p]\Sym^{n}\Hdr^1(E/K).
\]
One can choose generators
$\omega$ and $\eta$ for
the group $\Hdr^1(E/K)$, and it thus makes sense to evaluate
$\AJ_{K,p}(\Delta_\varphi)$ on an element of the form $f\wedge
\omega^j\eta^{n-j}$. By explicitly computing this map, and combining the result with the
formula in Equation~\eqref{eq:general}, we obtain the following result (see Corollary~\ref{cor:mainresult} for a more precise and general statement):
\begin{thm*}
There exist explicit isogenies $\varphi$ and $\ol\varphi$ as above and a constant $\Omega\in K^\times$ such that for all $0\leq j\leq n$:
\[
\AJ_{K,p}(\Delta_{\varphi}-\Delta_{\ol\varphi})(f\wedge\omega^j\eta^{n-j})=\Omega^{j-n}L_p'(f,K,j+1).
\]
\end{thm*}

This result is to be regarded as a $p$-adic
Gross-Zagier type formula for the anti-cyclotomic $p$-adic $L$-function. Note however that instead of heights it involves the
$p$-adic Abel-Jacobi map. It can also be seen as a generalization of
the main result of Iovita and Spie\ss{} in~\cite{iovita2003dpa} to all values
in the critical range.

This paper is structured as follows: in Section~\ref{section:RigAn} we recall concepts from rigid analytic geometry and $p$-adic integration. Section~\ref{section:ShimCurves} deals with the theory of Shimura curves, their $p$-adic uniformization, and modular forms defined on them. Section~\ref{section:filteredphiNmodules}  is a refresher of part of Fontaine's theory of filtered Frobenius monodromy modules and their relationship with semistable Galois representations. In Sections~\ref{section:HiddenStructures} and~\ref{section:CohomologyXGamma} we describe the extra structure of certain de Rham cohomology groups as done by Coleman and Iovita in~\cite{coleman2003hss}, and describe pairings in certain cases. In Section~\ref{section:LFunction} we recall the definition of the anti-cyclotomic $p$-adic $L$-function $L_p(f,K,s)$, and prove a formula for its derivative in terms of integration on the $p$-adic upper-half plane. In Section~\ref{section:Motive} we construct a Chow motive $\cD_n$ and compute its realizations. Section~\ref{section:Final} contains the main result of this paper, and Section~\ref{section:Conclusion} gives concluding remarks and points towards future directions of research.
\section{Rigid Geometry}
\label{section:RigAn}
In this section we describe the structure of rigid analytic space of the $p$-adic upper-half plane and its quotients by certain arithmetic subgroups of $\GL_2$. For more details on rigid analytic geometry the reader is invited to refer to~\cite{bosch1984naa} or \cite{fresnel2004rigid}. Here we will use the notation of this latter reference.

\subsection{The $p$-adic upper-half plane $\cH_p$ and its boundary}

Fix through this work a rational prime $p$, and denote by $\CC_p$ the topological completion of the algebraic closure of $\QQ_p$. On the projective line over $\CC_p$, denoted $\PP^1(\CC_p)$, we consider the strong G-topology, which is described as:
\begin{enumerate}
\item Every open subset $U$ of $\PP^1(\CC_p)$ is admissible for the strong G-topology, and
\item An open covering $\{U_i\}_{i\in I}$ of an open $U\subseteq\PP^1(\CC_p)$ is an admissible covering for the strong G-topology if for every affinoid $F\subseteq U$ there is a finite subset $J\subseteq I$ and affinoids $F_j\subseteq U_j$ for all $j\in J$, such that $F\subseteq U_{j\in J} F_j$.
\end{enumerate}

We proceed to define a certain analytic subspace of $\PP^1(\CC_p)$, the $p$-adic upper-half plane $\cH_p$ defined over $\QQ_p$. It can be defined as a formal scheme over $\ZZ_p$, but we are only interested in the rigid-analytic space associated to its generic fiber, which is a subset of $\PP^1(\CC_p)$, together with a  collection of affinoids that define its rigid-analytic structure. One can find more details of its construction in~\cite[Section 3]{dasgupta2008pau}.

Recall the Bruhat-Tits tree of $\GL_2(\QQ_p)$, as explained in~\cite{darmon2004rpm}.  It is a graph $\cT$ which has as set of vertices $\fV(\cT)$ the similarity classes of $\ZZ_p$-lattices in $\QQ_p^2$. Two vertices are connected by an edge whenever they have representative lattices $\Lambda_1$ and $\Lambda_2$ satisfying
\[
p\Lambda_2\subsetneq \Lambda_1\subsetneq \Lambda_2.
\]
The set of edges of $\cT$ will be denoted $\fE(\cT)$. Note that the above symmetrical relation makes $\cT$ an unoriented graph. In fact, $\cT$ is a $(p+1)$-regular tree; that is, each vertex has exactly $(p+1)$ neighbors. There is a natural action of $\PGL_2(\QQ_p)$ on $\cT$ by acting on the lattices, and this action respects the edges, yielding an action of $\PGL_2(\QQ_p)$ on $\cT$ by (continuous) graph automorphisms.

Fix an ordering of the edges of $\cT$, and denote by $\vec\fE(\cT)$ the set of ordered edges. If the ordered edge $e$ connects the vertices $v_1$ and $v_2$, we write $v_1=o(e)$ and $v_2=t(e)$. We also write $\bar e$ for the opposite edge, which has $o(\bar e)=v_2$ and $t(\bar e)=v_1$.

The Bruhat-Tits tree has a distinguished vertex, written $v_0$, which corresponds to the homothety class of the standard lattice $\ZZ_p^2$ inside $\QQ_p^2$. The edges $e$ with $o(e)=v_0$ correspond to the $(p+1)$ sublattices of index $p$ in $\ZZ_p^2$, which in turn are in bijection with the points $t$ of $\PP^1(\FF_p)=\{0,1,\ldots,p-1,\infty\}$. For such a point $t$, we denote the by $e_t\in \vec\fE(\cT)$ the corresponding edge. Since $\cT$ is locally-finite, it can be endowed with a natural topology, and it thus becomes contractible topological space. Given an edge $e\in\fE(\cT)$, we write $[e]\subset\cT$ for the closed edge (which contains the two vertices that $e$ connects), and $]e[$ for the corresponding open edge.

As a set, the space $\cH_p$ is defined as $\cH_p(\CC_p)\dfn\PP^1(\CC_p)\setminus \PP^1(\QQ_p)$. Note also that $\GL_2(\QQ_p)$ acts on $\cH_p(\CC_p)$ by fractional linear transformations. We describe a covering by basic affinoids and annuli, using the Bruhat-Tits tree defined above. Let
\[
\red\colon \PP^1(\CC_p)\to \PP^1(\Fbar_p)
\]
denote the natural map given by reduction modulo $\fm_{\cO_{\CC_p}}$, the maximal ideal of the ring of integers of $\CC_p$. Given a point $\tilde x$ in $\PP^1(\ol\FF_p)$, the \emphh{residue class}of $\tilde x$ is the subset $\red^{-1}(\{\tilde x\})$ of $\PP^1(\CC_p)$.

Define the set $A_0$ to be $\red^{-1}(\PP^1(\ol\FF_p)\setminus\PP^1(\FF_p))$. This is the prototypical example of a \emphh{standard affinoid}. Define also a collection of annuli $W_t$ for $t\in\PP^1(\FF_p)$, as:
\[
W_t\dfn \left\{\tau\in\PP^1(\CC_p)\stt \frac{1}{p}<\abs{\tau-t}<1\right\},\quad 0\leq t\leq p-1,\quad W_\infty\dfn\left\{\tau\stt 1<\abs{\tau}<p\right\}.
\]
Note that $A_0$ and the annuli $W_t$ are mutually disjoint. There is a  ``reduction map'' $r\colon \cH_p(\CC_p)\to \cT$: on the set
\begin{equation}
\label{eq:defX1}
X_1\dfn A_0\cup \left(\cup_{t\in\PP^1(\FF_p)} W_t\right)
\end{equation}
it is defined by:
\[
r(\tau)\dfn\begin{cases}
v_0&\If\tau\in A_0\\
e_t&\If\tau\in W_t,
\end{cases}
\]
where we recall that we labeled the edges with origin $v_0$ as $e_t$, with $t\in\PP^1(\FF_p)$. The map $r$ extends to all $\cH_p(\CC_p)$ by requiring it to be $\GL_2(\QQ_p)$-equivariant.

For each vertex $v\in\fV(\cT)$, let $\cA_v\dfn r^{-1}(\{v\})$. For each edge $e\in\fE(\cT)$, write $\cA_{[e]}\dfn r^{-1}([e])$ and $\cA_{]e[}\dfn r^{-1}(]e[)$. Then the collection $\{\cA_{[e]}\}_{e\in\fE(\cT)}$ gives a covering of $\cH_p(\CC_p)$ by standard affinoids, and their intersections are:
\[
\cA_{[e]}\cap\cA_{[e']}=
\begin{cases}
\emptyset&\If [e]\cap[e']=\emptyset,\\
\cA_v&\If [e]\cap [e']=\{v\}.
\end{cases}
\]
As an example, note that $\cA_{v_0}=A_0$ and that for $0\leq t\leq p-1$, the set $\cA_{[e_t]}$ is the union of two translates of $A_0$ glued along $W_t$. This covering gives the rigid-analytic space structure to $\cH_p$. Note that the nerve of this covering is precisely the Bruhat-Tits tree $\cT$.

The boundary of $\cH_p$ is the set $\PP^1(\QQ_p)$, which has been removed from $\PP^1(\CC_p)$ in order to obtain $\cH_p(\CC_p)$. An \emphh{end} of $\cT$ is an equivalence class of sequences $\{e_i\}_{i\geq 1}$ of edges $e_i\in\vec\fE(\cT)$, such that $t(e_i)=o(e_{i+1})$, and such that $t(e_{i+1})\neq o(e_i)$. Two such sequences are identified if a shift of one is the same as the other, for large enough $i$. Write $\fE_\infty(\cT)$ for the space of ends.

Choose once and for all an edge $e_0\in\vec\fE(\cT)$ such that its stabilizer inside $\PGL_2(\QQ_p)$ is the image of the unit group in the Eichler order
\[
R\dfn\left\{\smtx abcd \in M_2(\ZZ_p)\stt c\equiv 0\pmod{p}\right\}.
\]
The map $\beta\mapsto \beta\cdot e_0$ identifies $\PGL_2(\QQ_p)/\stab (e_0)$ with $\vec\fE(\cT)$. The inverse map will be written $e\mapsto \beta_e$. Also, the map
\[
N\colon\{e_i\}_i\mapsto \lim_i \beta_{e_i}(\infty)
\]
identifies $\fE_\infty(\cT)$ with $\PP^1(\QQ_p)$. In this way, the $p$-adic topology on $\PP^1(\QQ_p)$ induces a topology on $\fE_\infty(\cT)$. For an edge $e\in \vec\fE(\cT)$, write $U(e)$ for the compact open subset of $\fE_\infty(\cT)$ consisting of those ends having a representative which contains $e$. This is a basis for the topology, and we can compactify $\cT$ by adding to it its ends. Calling this completed tree $\cT^*$, we can extend $r$ to a map $r\colon\PP^1(\CC_p)\to \cT^*$.

\subsection{Quotients of \texorpdfstring{$\cH_p$}{Hp} by arithmetic subgroups}
In \cite{gerritzen1983schottky} one can find the general theory of Schottky groups, which are those groups that lead to manageable quotients of $\cH_p$. We will restrict our attention to a very special class of those groups, which are related to the $p$-adic uniformization of Shimura curves.

Let $B$ be a definite rational quaternion algebra of discriminant $N^-$ coprime to $p$. Fix an Eichler $\ZZ[\frac 1p]$-order $R$ of level $N^+$ in $B$, and fix an isomorphism $B\tns[\QQ]\QQ_p\cong M_2(\QQ_p)$. Let $\Gamma$ be the group of elements of reduced norm $1$ in $R$.

 The group $\Gamma$ is a discrete cocompact subgroup of $\SL_2(\QQ_p)$, as is shown in~\cite[Proposition~9.3]{shimura1994iat}. Suppose for simplicity that $\Gamma$ contains no elliptic points, and consider the topological quotient $\pi\colon \cH_p\to X_\Gamma\dfn\Gamma\backslash\cH_p$. Since $\Gamma$ is discrete, the space $X_\Gamma$ can be given a structure of rigid-analytic space in a way so that $\pi$ is a morphism of rigid-analytic spaces. An admissible covering is indexed by the quotient graph $\Gamma\backslash \cT$, in the same way that was done for $\cH_p$. In this way one obtains a complete curve called a \emphh{Mumford-Shottky curve}. In Section~\ref{section:ShimCurves} we will see how these curves are related to Shimura curves.

\subsection{\texorpdfstring{$p$}{p}-adic integration}
\label{sec:padicint}
The theory of $p$-adic integration was constructed initially by Coleman in \cite{coleman1989rlc}, \cite{coleman1985tpc} and \cite{coleman1982dra}, and further developed by Coleman-Iovita in \cite{coleman2003hss}, and by de Shalit \cite{deshalit1989eca}, among others. We will borrow very little from this vast theory, in order to cover only those concepts that are needed later.

Concretely, we will construct an integration theory on those rigid spaces which allow a covering by certain type of open subsets of $\PP^1(\CC_p)$. The $p$-adic upper-half plane $\cH_p$ admits such a covering, and hence we will obtain a theory of integration on $\cH_p$ and on Mumford-Schottky curves.

We will also be interested in the integration of general vector bundles over the curves above. It turns out, however, that the bundles that we will encounter in this work have a basis of horizontal sections, and therefore one can integrate component-wise, thus reducing to integration with trivial coefficients.

\begin{df}
\label{df:wideopen}
  A \emphh{wide open} is a set of the form
\[
U\dfn \{z\in\PP^1(\CC_p)\stt \abs{f(z)}<e_f,f\in S\},
\]
where $S$ a finite set of rational functions over $\CC_p$ containing at least one non-constant function, and $e_f\in\{1,\infty\}$ for all $f$.
\end{df}

For example, the open balls $B(a,r)$, the open annuli $\cA_{]e[}$ and the set $X_1$ as in Equation~\eqref{eq:defX1} are all wide open sets.

If $X\subseteq\PP^1(\CC_p)$ is an affinoid and $U\supset X$ is a wide open such that the complement $U\setminus X$ is a disjoint union of annuli, we say that $U$ is a  \emphh{wide open neighborhood} of $X$. A \emphh{basic wide open} is a set $C$ of the form:
\[
C=\AA^1(\CC_p)\setminus \bigcup_{a\in S\cup\{\infty\}} B[a,r_a],
\]
where $S$ is a finite subset of $\BB^1$ no two elements of which are contained in the same residue class, and for each $a\in S\cup\{\infty\}$ the radius $r_a\in \abs{\CC_p}$ satisfies $\abs{r_a}<1$. It can also be written as the disjoint union of a connected affinoid $X$ (which is a full subspace of $\BB^1$), and of $|S|+1$ wide open annuli $V_a$:
\[
X=\BB^1\setminus \bigcup_{a\in S} B(a,1),
\]
\[
V_a=A(a,r_a,1),\quad a\in S\cup\{\infty\}.
\]
Note that the set $X_1$ appearing in Equation~\eqref{eq:defX1} is an example of a basic wide open, with $X(X_1)=A_0$.

A locally analytic homomorphism $l\colon\CC_p^\times\to\CC_p^+$ such that $\frac{d}{dz}l(1)=1$ is called a \emphh{branch of the logarithm}. It can be easily shown that $l(z)$ is analytic on $B(x,|x|)$, for all $x\in\CC_p^\times$.

Given any open subset $X$ of $\PP^1(\CC_p)$, we will denote by $\cO(X)$ (resp. $\cL(X)$) the ring of rigid analytic functions (resp. of locally analytic functions) on $X$. Coleman defines in~\cite{coleman1982dra} the notion of logarithmic $F$-crystal on $C$ for any basic wide open $C$. This is a certain type of $\cO(C)$-submodule of $\cL(C)$ satisfying several technical conditions (see \cite[page 184, conditions A-F]{coleman1982dra}). One of these conditions is the uniqueness property: if $M$ is a logarithmic $F$-crystal on $C$, then any element $f\in M$ that vanishes in a non-empty open subset of $C$ must be zero. This notion includes as basic example the ring $\cO(C)$ itself.

For the following result, one needs to define $\cO_{\Log}$ and $\Omega_M$, whose definition can be found in page 177 and page 179 of~\cite{coleman1982dra}. Let $\phi$ be a Frobenius morphism on $X$. Coleman proves the following result:

\begin{lemma}
Let $M$ be a logarithmic $F$-crystal on a basic wide open space $C$, and let $\omega$ be an element of $\Omega_M(C)$. There exists a locally-analytic function $F_\omega\in\cL(C)$, unique up to an additive constant, which satisfies:
  \begin{enumerate}
  \item $dF_\omega=\omega$,
  \item There is a wide open neighborhood $V$ of $C$ such that $\phi^*F_\omega-bF_\omega\in M(V)$, for some $b\in\CC_p$ which is not a root of unity, and
  \item The restriction of $F_\omega$ to the underlying affinoid $X$ is analytic in each residue class of $X$, and the restriction to $V_a$ is in $\cO_\Log(V_a)$ for all $a\in S$.
  \end{enumerate}
\end{lemma}

Given a logarithmic $F$-crystal $M$, one can then define an $\cO(C)$-module $M'$ as follows:
\[
M'\dfn M+\sum_{\omega\in\Omega_M(C)} F_\omega\cO(C).
\]

This is the unique minimal logarithmic $F$-crystal on $C$ which contains $M$ and such that $dM'\supseteq \Omega_M(C)$. In particular since $\cO(C)$ is a logarithmic $F$-crystal, one can define $A^1(C)$ to be $\cO(C)'$, and obtain:
\begin{thm}[(Coleman)]
\label{thm:coleman-integral}
  Let $\omega\in\Omega^1(C)$. There exists a unique (up to constants) function $F_\omega\in A^1(C)$, such that $dF_\omega=\omega$.
\end{thm}
\begin{proof}
See \cite[Theorem~5.1]{coleman1982dra}.
\end{proof}

Let $Y$ be a rigid-analytic space which can be covered by a family $\cC$ of basic wide opens, which intersect at basic wide opens, and such that the nerve of the covering is simply connected. Let $\cA^1$ be the sheaf of $\cO_Y$-modules defined by $\cA^1(U)\dfn A^1(U)$ for each $U\in\cC$. The following corollary is an easy consequence of the results stated so far:

\begin{cor}
There is a short exact sequence:
\[
0\to \CC_p\to H^0(Y,\cA^1)\tto{d} H^0(Y,\cA^1)\tns[\cO_Y(Y)]\Omega^1(Y)\to 0
\]
\end{cor}
\section{Shimura Curves}
\label{section:ShimCurves}
In this section we introduce the different ways in which Shimura curves appear in this work.
\subsection{Moduli spaces}
\label{subsection:shimuracurves}
A good exposition of the theory of Shimura curves and their $p$-adic uniformization can be found in~\cite{boutot1991upa}. Here we just recall the basic facts.

Fix for the rest of the paper an integer $N$ which can be factored as $N=pN^-N^+$, where $p$ is a prime which will remain fixed, $N^-$ is a positive squarefree integer with an odd number of prime divisors none of which equals $p$, and $N^+$ is a positive integer relatively prime to $pN^-$. Let $\cB$ be the indefinite rational quaternion algebra of discriminant $pN^-$. Fix a maximal order $\Rmax$ in $\cB$, and an Eichler order $\cR$ of level $N^+$ contained in $\Rmax$.

\begin{df}
Let $S$ be a $\QQ$-scheme. An \emphh{abelian surface with quaternionic multiplication} (by $\Rmax$) and level $N^+$-structure over $S$ is a triple $(A,i,G)$ where
\begin{enumerate}
\item $A$ is a (principally polarized) abelian scheme over $S$ of relative dimension $2$;
\item $i\colon \Rmax\injects \End_S(A)$ is a a ring homomorphism;
\item $G$ is a subgroup scheme of $A$ which is \'etale-locally isomorphic to $(\ZZ/N^+\ZZ)^2$ and is stable and locally cyclic under the action of $\cR$.
\end{enumerate}
When no confusion may arise, such a triple will be called an \emphh{abelian surface with QM}.
\end{df}

\begin{df}
\label{df:ShimuraCurve}
The \emphh{Shimura curve} $X\dfn X_{N^+,pN^-}/\QQ$ is the \emph{coarse} moduli scheme representing the moduli problem over $\QQ$:
\[
S\mapsto\{\text{ isomorphism classes of abelian surfaces with QM over $S$ }\}.
\]
\end{df}

\begin{prop}[Drinfel'd]
The Shimura curve $X_{N^+,pN^-}$ is a smooth, projective and geometrically connected curve over $\QQ$.
\end{prop}
\begin{proof}
  See~\cite[Section III]{boutot1991upa}.
\end{proof}

We will need to work with a Shimura curve which is actually a fine moduli space. For that, we need to rigidify the moduli problem, as follows.

\begin{df}
Let $M\geq 3$ be an integer relatively prime to $N$. Let $S$ be a $\QQ$-scheme. An \emphh{abelian surface with QM and full level $M$-structure} (QM by $\Rmax$ and level $N^+$-structure is understood) is a quadruple $(A,i,G,\overline\nu)$ where $(A,i,G)$ is as before, and $\overline\nu\colon (\Rmax/M\Rmax)_S\to A[M]$ is a $\Rmax$-equivariant isomorphism from the constant group scheme $(\Rmax/M\Rmax)_S$ to the group scheme of $M$-division points of $A$.
\end{df}

\begin{df}
The Shimura curve $X_M=X_{N^+,pN^-,M}$ is defined to be the fine moduli scheme classifying the abelian surfaces with QM and full level $M$-structure.
\end{df}

\begin{rmk}
The curve $X_M$ is still smooth and projective over $\QQ$. However, it is not geometrically-connected. In fact, as we will see below, it is the disjoint union of $\#(\ZZ/M\ZZ)^\times$ components.
\end{rmk}

Forgetting the level $M$-structure yields a Galois covering $X_M\to X$ with Galois group
\[
(\Rmax/M\Rmax)^\times/\{\pm 1\}\cong \gl_2(\ZZ/M\ZZ)/\{\pm 1\}.
\]

\subsection[Heegner points]{Heegner points}
\label{subsection:HeegnerPoints}
Let $F$ be a field of characteristic zero.

\begin{df}
An abelian surface $A$ defined over  $F$ (with $i\colon \Rmax\injects \End_F(A)$ and
level-$N$ structure) is said to have \emphh{complex multiplication}
 (CM) if $\End_{\Rmax}(A)\neq \ZZ$.  In that case, 
$\cO\dfn \End_{\Rmax}(A)$ is an order in an imaginary quadratic
number field $K$, and one says that $A$ has CM by $\cO$.
\end{df}

\begin{df}
  A point on the Shimura curve $X_M$ is called a \emphh{Heegner point} if it
can be represented by a quadruple $(A,i,G,\ol\nu)$ such that $A$ has
complex-multiplication by $\cO$ and $G$ is $\cO$-stable. If we drop
the condition of $G$ being $\cO$-stable, then we call it a \emphh{CM
point}.
\end{df}

\begin{rmk}
\label{rmk:aisee} Suppose that $A$ has QM by $\Rmax$ and CM by $\cO_K$. Suppose that $\cO_K$ splits $\Rmax$. Then:
\[
\End_F(A)\cong \cO_K\tns \Rmax\cong M_2(\cO_K).
\]
By $\End_F(A)$ we mean the endomorphisms of $A$ as an algebraic variety over $F$. Fixing an isomorphism $\End(A)\cong M_2(\cO_K)$ yields an
isomorphism $A\cong E\times E$, where $E$ is an elliptic curve defined over
$H$, the Hilbert class field of $F$, with $\End_H(E)\cong \cO_K$. Explicitly, one can obtain each of the two copies of $E$ by applying to $A$ the endomorphism corresponding to the matrices
\[
\mtx 1000 \quad\text{ and } \mtx 0001.
\]
In
particular, $E$ is an elliptic curve with complex multiplication.
\end{rmk}

\subsection{\texorpdfstring{$p$}{p}-adic uniformization}
\label{subsection:padicuniformization}
We will use a uniformization result due to \v{C}erednik and Drinfel'd, which gives an explicit realization of the Shimura curves $X$ and $X_M$ as quotients of the $p$-adic upper-half plane. Let $B$ be the \emph{definite} rational quaternion algebra of discriminant $N^-$, and let $R$ be an Eichler $\ZZ[\frac 1 p]$-order of level $N^+$ in $B$. Define the group
\[
\Gamma\dfn\{ x\in R^\times \stt \nrd(x)=1\}.
\]
Fix an isomorphism
\[
\iota_p\colon B_p=B\tns[\QQ]\QQ_p\tto{\sim} M_2(\QQ_p).
\]

\begin{prop}
The isomorphism $\iota_p$ identifies the group $\Gamma$ with a discrete co-compact subgroup of $\SL_2(\QQ_p)$.
\end{prop}
\begin{proof}
  See~\cite[Proposition~9.3]{shimura1994iat}.
\end{proof}

The previous proposition makes it possible to consider the quotient $X_\Gamma\dfn \Gamma\backslash\cH_p$. The celebrated result of \v{C}erednik-Drinfel'd gives a deep relationship of the Shimura curves $X$ and $X_M$ defined above, with $X_\Gamma$.

\begin{thm}[(\v{C}erednik-Drinfel'd)]
\label{thm:cer-drin}
There is an isomorphism of rigid-analytic varieties:
\[
(X_{\QQ_p^\ur})^{\text{an}}\cong X_{\Gamma}\dfn \Gamma\backslash \cH_p.
\]
Moreover, for any integer $M\geq 3$, let $\Gamma_M$ be the subgroup of units of reduced norm congruent to $1$ modulo $M$. There is an isomorphism of rigid-analytic varieties:
\[
(X_M)^\text{an}_{\QQ_p^\text{ur}}\cong \Gamma\backslash \left(\cH_p\times (R/MR)^\times\right)\cong \coprod_{(\ZZ/M\ZZ)^\times} \Gamma_M\backslash \cH_p,
\]
which exhibits $X_M$ as a disjoint union of Mumford curves, and hence it is semistable.
\end{thm}
\begin{proof}
  Although the result is original of \v{C}erednik and Drinfel'd, a more detailed exposition of the proof can be found in~{\cite[Chap. III, 5.3.1]{boutot1991upa}}.
\end{proof}
\subsection{Modular forms}
\label{sec:modul-forms-shim}

Let $n\geq 0$ be an even integer. We want to explain the different ways of identifying modular forms with sections of certain sheaves associated to the Shimura curve $X\dfn X_{N^+,pN^-}$ as in Definition~\ref{df:ShimuraCurve}. Let $\cB, \cR^{\text{max}}, \cR$ be as in the definition of $X$, and choose another Eichler order $\tilde\cR\subseteq \cR$ with the property that the group of units of norm one $\tilde\cR^\times_1$ is free. Let $\tilde X$ be the Shimura curve associated to the order $\tilde\cR$, and let $G$ be the finite group $\cR^\times_1/\tilde\cR^\times_1$.

\begin{df}
Let $K$ be a field of characteristic $0$.A \emph{modular form of weight $n+2$ on $X$ defined over $K$} is a global section of the sheaf $\Omega_{\tilde X_K/K}^{\tns \frac{n+2}{2}}$ on $X_K$ which is invariant under the action of $G$. We denote by $M_{n+2}(X,K)$ the space of such modular forms.
\end{df}

\begin{rmk}
A simple argument shows that this definition does not depend on the choice of the auxiliary Eichler order $\tilde\cR$.
\end{rmk}

Let $K$ be either $\QQ_p^\ur$ or any complete field contained in $\CC_p$ which contains $\QQ_{p^2}$. Using the result of \v{C}erednik-Drinfel'd stated in Theorem~\ref{thm:cer-drin} we can give a more concrete description of $M_{n+2}(X,K)$.

\begin{df}
  A $p$-adic modular form of weight $n+2$ for $\Gamma$ is a rigid analytic function $f\colon\cH_p(\CC_p)\to\CC_p$, defined over $K$, such that
\[
f(\gamma z)=(cz+d)^{n+2}f(z),\qquad\text{ for all } \gamma=\smtx abcd \in \Gamma.
\]

Denote the space of such $p$-adic modular forms by $M_{n+2}(\Gamma)=M_{n+2}(\Gamma,K)$.
\end{df}

\begin{prop}
There is a canonical isomorphism
\[
M_{n+2}(\Gamma,K)\tto{\sim} M_{n+2}(X,K),
\]
which maps $f$ to $\omega_f\dfn f(z)dz^{\tns \frac{n+2}{2}}$.
\end{prop}

\subsection[Jacquet-Langlands]{The Jacquet-Langlands correspondence}
\label{section:jacquetlanglands}
In order to justify our interest in modular forms over Shimura curves, we would like to relate them to more familiar objects. Let $\TT$ be the abstract Hecke algebra generated by the Hecke operators $T_\ell$ for $\ell \nmid N$ and $U_\ell$ for $\ell\mid N$. The Hecke algebra $\TT$ acts naturally on the space $M_{n+2}(X,K)$, on which also act the Atkin-Lehner involutions.

\begin{thm}[(Jacquet-Langlands)]
\label{thm:jacquet-langlands}
  Let $K$ be a field. There is an isomorphism
\[
M_{n+2}(X,K)\tto{\sim} S_{n+2}(\Gamma_0(N),K)^{pN^- \text{-new}},
\]
which is compatible with the action of $\TT$ and the Atkin-Lehner involutions on each of the spaces.
\end{thm}

Therefore to a classical modular $pN^-$-new eigenform $f_\infty$ on the modular curve $X_0(N)$, there is associated an eigenform $f$ on the Shimura curve $X$, which is unique up to scaling. In Section~\ref{section:LFunction} we will see the construction of a $p$-adic $L$-function attached to $f$ which interpolates special values of the classical $L$-function associated to $f_\infty$.

\section[Filtered \texorpdfstring{$\protect(\phi,N)$}{phi-N}-modules]{Filtered \texorpdfstring{$(\phi,N)$}{phi-N}-modules}
\label{section:filteredphiNmodules}
Let $K$ be a field of characteristic $0$, which is complete with respect to a discrete valuation and has perfect residue field $\kappa$, of characteristic $p>0$. Let $K_0\subseteq K$ be the maximal unramified subfield of $K$. Concretely, $K_0$ is the fraction field of the ring of Witt vectors of $\kappa$. Let $\sigma\colon K_0\to K_0$ be the absolute Frobenius automorphism.

Consider the category of filtered Frobenius monodromy modules (or $(\phi,N)$-modules) over $K$, denoted by $MF_K^{(\phi,N)}$. Its objects are quadruples $(D,\Fil^\bullet,\phi,N)$, where $D$ is a finite dimensional $K_0$-vector space, $\Fil^\bullet=\Fil^\bullet_D$ is an exhaustive and separated decreasing filtration on the vector space ${D_K\dfn D\tns[K_0] K}$ over $K$ (called the \emphh{Hodge filtration}), $\phi=\phi_D\colon D\to D$ is a $\sigma$-linear automorphism, called the \emphh{Frobenius on $D$}, and $N=N_D\colon D\to D$ is a $K$-linear endomorphism, called the \emphh{monodromy operator}, which satisfying $N\phi=p\phi N$. Sometimes we write $D$ to refer to the tuple $(D,\Fil_D^\bullet,\phi_D,N_D)$. For a precise definition of this category refer to~\cite{fontaine1994cpp} or to the lecture notes~\cite{brinon-p}.

Forgetting the monodromy action or, equivalently, setting $N=0$, gives a full subcategory of $\MF_K^{(\phi,N)}$, called the category of \emphh{filtered F-isocrystals over $K$}. The full subcategory obtained by additionally forgetting the filtration is the category of \emphh{isocrystals over $K_0$}, which were studied and classified by Dieudonn\'e and Manin.

The category $MF_K^{(\phi,N)}$ is an additive tensor category which admits kernels and cokernels. Also, if $D=(D,F_D^\bullet,\phi_D,N_D)$ is a filtered $(\phi,N)$-module, and $j$ is an integer, we define another filtered $(\phi,N)$-module $D[j]$, the $j$th \emphh{Tate twist} of $D$, as $D[j]=(D,F_D^{\bullet-j},p^j\phi_D,N_D)$, where by $F_D^{\bullet-i}$ we mean:
\[
F^i(D[j]_K)=F^{i-j}(D_K),\quad\text{ for all } i\in\ZZ.
\]

Consider the category $\Rep_{\QQ_p}(G_K)$ of $p$-adic representations of $G_K$, whose objects are finite-dimensional $\QQ_p$-vector-spaces with a continuous linear $G_K$-action. It is an abelian tensor category, with twists given by tensoring with powers of the Tate representation $\QQ_p(1)\dfn(\varprojlim_n\mmu_{p^n})\tns[\ZZ_p]\QQ_p$.

The functors $D_\st$ and $V_\st$ of Fontaine, constructed originally in~\cite{fontaine1994cpp}, are functors relating the category of $p$-adic representations of $G_K$ with that of filtered Frobenius monodromy modules:
\[
\xymatrix{
\Rep_{\QQ_p}(G_K)\ar@<+3pt>[r]^{\Dst}&\MF_K^{(\phi,N)},\ar@<+3pt>[l]^{\Vst}
}
\]
\begin{df}
A $p$-adic representation $V$ of $G_K$ is \emph{semistable} if the canonical injective map:
\[
\alpha\colon \Dst(V)\tns[K_0] \Bst=(V\tns[\QQ_p]\Bst)^{G_K}\tns[K_0]\Bst\injects (V\tns[\QQ_p]\Bst)\tns[K_0]\Bst\tto{\id\tns m} V\tns[\QQ_p]\Bst
\]
is surjective. The category of semistable representations, denoted $\Rep_\st(G_K)$ is the full subcategory of $\Rep_{\QQ_p}(G_K)$ of semistable objects.
\end{df}

\begin{rmk}
Let $X/K$ be a proper variety with a semi-stable model. Consider the \'etale cohomology groups:
\[
\Het^i(\ol X,\QQ_p)\dfn\left(\varprojlim_{n} \Het^i(\ol X,\ZZ/p^n\ZZ)\right)\tns[\ZZ_p]\QQ_p.
\]
These vector spaces are naturally finite-dimensional continuous $G_K$-representations. Results of Fontaine-Messing, Hyodo-Kato, Faltings and Tsuji imply that these representations are semistable. They constitute in fact the main source of semistable representations.
\end{rmk}

\begin{df}
A filtered $(\phi,N)$-module $D$ is \emph{admissible} if it is isomorphic to $\Dst(V)$ for some semistable representation $V$ of $G_K$. The full subcategory of admissible filtered $(\phi,N)$-modules is denoted $\MF_K^{\ad,(\phi,N)}$.
\end{df}

\begin{thm}[({\cite[Theorem A]{MR1779803}})]
The functors $\Dst$ and $\Vst$ give an equivalence of categories between $\Rep_\st(G_K)$ and $\MF_K^{\ad,(\phi,N)}$, which is compatible with exact sequences, tensor products and duality.
\end{thm}

The main use that we have for this fact is the following:
\begin{cor}
  Let $V,W$ be two objects in $\Rep_\st(G_K)$. The functors $\Dst$ and $\Vst$ induce a canonical group isomorphism
\[
\Ext^1_{\Rep_\st(G_K)}(V,W)\cong \Ext^1_{\MF_K^{\ad,(\phi,N)}}\left(\Dst(V),\Dst(W)\right),
\]
where $\Ext_{\cC}^1$ denotes the extension-group bifunctor in the category $\cC$.
\end{cor}

Next we study the extensions of filtered $(\phi,N)$-modules. Let $D$ be an object in this category. Given a rational number $\lambda=r/s$, where $r,s\in\ZZ$ are such that $(r,s)=1$ and $s>0$, define $D_\lambda$ to be the largest subspace of $D$ which has an $\cO_{K_0}$-stable lattice $M$ satisfying $\phi^s(M)=p^rM$. The subspace $D_\lambda$ is called the \emphh{isotypical component} of $D$ of slope $\lambda$. The \emph{slopes} of $D$ are the rational numbers $\lambda$ such that $D_\lambda\neq 0$, and $D$ is called \emph{isotypical} of slope $\lambda_0$ if $D=D_{\lambda_0}$. The Dieudonn\'e-Manin classification gives a decomposition of isocrystals by slopes:
\[
D=\bigoplus_{\lambda\in \QQ} D_\lambda.
\]
Note also that $N(D_\lambda)\subseteq D_{\lambda-1}$ for all $\lambda\in \QQ$. The following result appears in \cite[Lemma 2.1]{iovita2003dpa}, although its proof is mostly omitted. For completeness, we present here a fully detailed proof.

\begin{lemma}
\label{lemma:easylemma}
Let $D$ be a filtered $(\phi,N)$-module, $n$ an integer, and assume that $N$ induces an isomorphism between the isotypical components $D_{n+1}$ and $D_{n}$. Then there is a canonical isomorphism
\[
\Ext^1_{\MF_K^{(\phi,N)}}(K[n+1],D)\cong  D/\Fil^{n+1}D,
\]
that maps the class of an extension
\[
0\to D \tto{\iota} E\tto{\pi} K[n+1]\to 0
\]
to $(s_1(1)-s_2(1))+\Fil^{n+1}D_K$, where:
\begin{enumerate}
\item $s_1\colon K[n+1]\to E$ is a splitting of $\pi$ which is compatible with the Frobenius and monodromy operator, but not necessarily with filtrations, and
\item $s_2\colon K[n+1]\to E$ is  splitting of $\pi$ compatible with the filtrations, but not necessarily with the Frobenius and monodromy operators.
\end{enumerate}
\end{lemma}
\begin{rmk}
  The fact that the splittings $s_1$ and $s_2$ exist is part of the statement of the lemma.
\end{rmk}

\begin{proof}
First, note that by applying the snake lemma to the following diagram with exact rows:
\[
\xymatrix{
0\ar[r]& \Fil^{n+1}D_K\ar[r]\ar@{^{(}->}[d]&\Fil^{n+1}E_K\ar[r]\ar@{^{(}->}[d]&\Fil^{n+1} K\ar@{=}[d]\ar[r]&0\\
0\ar[r]& D_K\ar[r]&E_K\ar[r]&K\ar[r]&0,
}
\]
we get an isomorphism $D_K/\Fil^{n+1}D_K\cong E_K/\Fil^{n+1}E_K$, and hence we just need to find an element in $E_K/\Fil^{n+1}E_K$. Explicitly, once we get $s_1(1)\in E_K$, we can consider $s_1(1)-s_2(1)$, where $s_2$ is a splitting of the extension which is compatible with the filtrations. Such a splitting $s_2$ exists because the category of $K$-vector spaces is semisimple. Since $\pi(s_1(1)-s_2(1))=0$, we can view $s_1(1)-s_2(1)$ as an element of $D_K$ (via $\iota$), thus making the isomorphism explicit.

The filtered $(\phi,N)$-module $K[n+1]$ is pure of slope $n+1$, and the hypothesis on the monodromy action $N$ on $D$ gives a commutative diagram with exact rows:
\[
\xymatrix{
0\ar[r]&D_{n+1}\ar[r]\ar[d]^{N}_{\cong}&E_{n+1}\ar[d]^{N}\ar[r]&K_0\ar[r]\ar[d]&0\\
0\ar[r]&D_{n}\ar[r]&E_{n}\ar[r]&0\ar[r]&0.
}
\]
An application of the snake lemma and the fact that the left vertical arrow is an isomorphism yields another isomorphism
\[
\pi_|\colon \ker\left(E_{n+1}\tto{N} E_{n}\right)\tto{\sim} K_0,
\]
and we define $s_1\colon K\to E_K$ as its inverse. Then $s_1$ is compatible with the action of $\phi$ and $N$, by construction.

We check that the assignment of $s_1(1)+\Fil^{n+1}E_K$ to an extension $0\to D\to E\to K[n+1]\to 0$ is well-defined: if the extension is trivial, then $s_1$ can be chosen to be compatible with $\Fil$, and we then get
\[
s_1(1)\in s_1(\Fil^n K[n+1])\subseteq \Fil^{n+1} E_K.
\]

Conversely, given $d+\Fil^{n+1}D_K\in D_K/\Fil^{n+1}D_K$, we construct a filtered $(\phi,N)$-module $E^{(d)}$ as an extension of $K[n+1]$ by $D$. We define $E^{(d)}_0=D_0\oplus (K_0[n+1])$, as $(\phi,N)$-modules. The filtration on $E^{(d)}_K=E^{(d)}_0\tns[K_0]K$ is defined as follows:
\[
\Fil^j E^{(d)}_K\dfn\left\{(x,t)\in D_K\oplus K\stt t\in \Fil^{j-n-1}K,\ x+td\in\Fil^jD\right\}.
\]

Consider the isomorphism class of the extension
\[
\Xi\colon\quad 0\to D\tto{\iota} E^{(d)}\tto{\pi} K[n+1]\to 0,
\]
where the map $\iota$ is the canonical inclusion, and the map $\pi$ is the canonical projection. Note that this sequence is exact and well defined, since
\[
\pi(\Fil^jE^{(d)}_K)=\Fil^{j-n-1}K=\Fil^j K[n+1].
\]
Moreover, if $d\in \Fil^{n+1}D_K$, then the map
\[
1\mapsto (0,1)
\]
splits the extension $\Xi$ in the category of filtered $(\phi,N)$-modules. Hence the map
\[
D_K/\Fil^{n+1}D_K\to \Ext^1(K[n+1],D) 
\]
which assigns the extension $\Xi$ to $d\in D/\Fil^{n+1}D$ is well defined.

To end the proof, we need to check that the two assignments are mutually inverse. Starting with $d+ \Fil^{n+1}D_K$, the vector space splitting $1\mapsto (0,1)$ is compatible with the Frobenius and monodromy actions. Also the vector space splitting $1\mapsto (-d,1)$ is compatible with the filtrations. We obtain the class of $d$ in $D_K/\Fil^{n+1}D_K$, as wanted.

Conversely, start with an arbitrary extension
\[
0\to D\tto{\iota} E\tto{\pi} K[n+1]\to 0.
\]
Choose $s_1$ and $s_2$ two splittings of $\pi$ as before, and define $d\in D_K$ such that $\iota_K(d)=s_1(1)-s_2(1)$. Consider now the map $E^{(d)}\to E$ sending
\[
(x,t)\mapsto \iota(x)+s_1(t)=\iota(x+td)+s_2(t).
\]
The first expression shows that this is a map of $(\phi,N)$-modules. The second expression shows that it respects the filtrations. Its inverse is the map
\[
y\mapsto \left(\iota^{-1}(y-s_1(\pi(y))),\pi(y)\right)=\left(\iota^{-1}(y-s_2(\pi(y)))-\pi(y)d,\pi(y)\right).
\]
Again, the first expression shows that it is respects the Frobenius and monodromy actions, while the second shows that it respects the filtrations. This concludes the proof.
\end{proof}

Let $\kappa$, $K_0$ and $K$ be as before. Let $Z$ be a formal $\cO_K$-scheme. The previous constructions can be extended to formal schemes, in a way that one recovers filtered Frobenius monodromy modules as the situation $Z=\Spf\cO_K$. More details can be found in the exposition of~\cite{iovita2003dpa}.

\begin{df}
An \emph{enlargement} of $Z$ is a pair $(T,z_T)$ consisting of a flat $p$-adic formal $\cO_K$-scheme $T$ and a morphism of formal $\cO_K$-schemes $z_T\colon T_0\to Z$ (where $T_0$ is closed subscheme of $T$ defined by the ideal $p\cO_T$, with the reduced scheme structure).

A \emphh{morphism of enlargements} of $Z$, say $(T',z_{T'})\to (T,z_T)$ is an $\cO_K$-morphism $g\colon T'\to T$ such that $z_T\circ g_0=z_{T'}$.
\end{df}

\begin{df}
A \emphh{convergent isocrystal} $\cE$ on $Z$ is the following data:
\begin{enumerate}
\item For every enlargement $T=(T,z_T)$ of $Z$, a coherent $\cO_T\tns[\cO_K] K$-module $\cE_T$.
\item For every morphism of enlargements $g\colon (T',z_{T'})\to (T,z_T)$, an isomorphism of $\cO_{T'}\tns[\cO_K]K$-modules
\[
\theta_g\colon g^*(\cE_T)\to \cE_{T'},
\]
such that the collection $\{\theta_g\}$ satisfies the cocycle condition.
\end{enumerate}
\end{df}

Let $\sigma\colon W(\kappa)\to W(\kappa)$ be the Frobenius automorphism on the ring of Witt vectors of $\kappa$, which can be lifted to the absolute Frobenius $F\colon Z\to Z^\sigma$. Given an enlargement $(T,z_T)$ of $Z$, the pair $(T,F\circ z_T)$ is an enlargement of $Z^\sigma$ and hence $(T^{\sigma^{-1}},(F\circ z_T)^{\sigma^{-1}})$ is an enlargement of $Z$. Given an isocrystal $\cE$, define $F^*\cE$ as the isocrystal which on $Z$ assigns to $(T,z_T)$ the $\cO_T\tns[\cO_K]K$-module:
\[
\alpha(\sigma)_* \cE_{(T^{\sigma^{-1}},(F\circ z_T)^{\sigma^{-1}})}.
\]

\begin{df}
A \emphh{convergent $F$-isocrystal} on $Z$ is a convergent isocrystal $\cE$ on $Z$ together with an isomorphism of crystals $\Phi\colon F^*\cE\to \cE$.
\end{df}

Assume from now on that $Z$ is analytically smooth over $\cO_K$. Associated to $\cE$ there is a coheren $\cO_{Z^\an}$-module $E^{\text{an}}=E_Z^{\text{an}}$, and in this case one may consider the Gauss-Manin connection $\nabla$, which is a natural connection on $E^\text{an}$ defined as a certain connecting homomorphism in the Hodge to de Rham spectral sequence for $Z$. A precise definition can be found in~\cite{katz1968drc}, and the required facts about its properties can be found in~\cite{ogus51f}.

\begin{df}
A \emphh{filtered convergent $F$-isocrystal} on $Z$ consists of a convergent $F$-isocrystal $\cE$ together with an exhaustive and separated decreasing filtration $\Fil^\bullet E^\text{an}$ of coherent $\cO_{Z^\text{an}}$-submodules, satisfying the following compatibility with respect to the Gauss-Manin connection $\nabla$:
\[
\nabla(\Fil^iE^\an)\subseteq (\Fil^{i-1}E^\an)\tns[\cO_Z^\text{an}] \Omega^1_{Z^\text{an}}.
\]
This condition is called \emphh{Griffiths' transversality} and is required in order to be able to define a filtration of the de Rham cohomology with coefficients in $E^\an$.
\end{df}

The category of filtered convergent $F$-isocrystals on $Z$ is an additive tensor category.

\begin{ex}
\label{ex:isocrystals}
\begin{enumerate}
\item\label{ex:trivial-isocrystal} The identity object $\cO_Z$ in this category is the assignment $T\mapsto \cO_T\tns K$. The Frobenius is the canonical one. The Gauss-Manin connection in this case is the trivial one, given by the usual derivation $d$. The filtration is given by
\[
\Fil^i=\begin{cases}
\cO_{Z^\text{an}}&\If i\leq 0\\
0&\Else .
\end{cases}
\]

\item\label{ex:derhamisocrystal} Let $f\colon X\to Z=\Spf(\cO_K)$ be smooth proper morphism of $p$-adic formal schemes. One can define an $F$-isocrystal $R^qf_*\cO_{X/K}$ using crystalline cohomology sheaves tensored with $K$. This is a filtered convergent $F$-isocrystal in a natural way: its analytification $(R^qf_*\cO_{X/K})^\text{an}$ is a coherent $\cO_{Z^\text{an}}$-module isomorphic to the relative de Rham cohomology $\cHdr^q(X^\text{an}/Z^\text{an})$, and the connection is the Gauss-Manin connection $\nabla$. The filtration is given by the Hodge filtration, induced from the Hodge to de Rham spectral sequence, as explained in~\cite{katz1968drc}.
\end{enumerate}
\end{ex}

\section{Hidden structures on the de Rham cohomology}
\label{sec:filfrobmonodromy}
\label{section:HiddenStructures}
This section recalls some of the results of \cite{coleman2003hss}. Let $\fX\to \Spec(\cO_K)$ be a proper semistable curve with connected fibers. Suppose that its generic fiber $X$ is smooth and projective, that the irreducible components $C_1,\ldots,C_r$ of the special fiber $C$ are smooth and geometrically connected, and that there are at least two of them. Assume also that the singular points of $C$ are $\kappa$-rational ordinary double points.

Denote by $X^\an$ the rigid analytification of $X$. We want to describe an admissible covering of $X^\an$. Consider the special fiber $C$ of $X$, and let $\fG=(\fV(\fG),\vec\fE(\fG))$ be the (oriented) intersection graph of $C$: there is one vertex for each irreducible component $C_i$, and the oriented edges are triples $e=(x,C_i,C_j)$, where $x$ is a singular point of $C$, and $C_i$ and $C_j$ are the two components on which $x$ lies. We set $o(e)=C_i$ and $t(e)=C_j$, and write $\bar e$ for the opposite edge $(x,C_j,C_i)$.

For each vertex $v=C_i$ of $\fG$, let $U_v\dfn \red^{-1}(C_i)$ be the tube associated to it. Here $\red\colon X^\an\to C(\bar k)$ is the reduction map. For each edge $e=(x,C_i,C_j)$, let $A_e$ be the wide open annulus $\red^{-1}(x)=U_{o(e)}\cap U_{t(e)}$, together with the orientation given by $e$. The sets $U_v$ give an admissible covering of $X^\an$. Define an involution $\ol{(\cdot)}$ on $\vec\fE(\fG)$ which maps an edge $e=(x,C_i,C_j)$, to
\[
\ol e\dfn(x,C_j,C_i).
\]
Write $\fE(\fG)$ for the set of unoriented edges of $\fG$, which can be thought as the set of equivalence classes of $\vec\fE(\fG)$ by this involution.

Let $E$ be a coherent locally free sheaf of $\cO_X$-modules, with a connection
\[
\nabla\colon E\to E\tns[\cO_X] \Omega^1_X,
\]
and filtration $\Fil^\bullet E$ by $\cO_X$-submodules satisfying Griffiths transversality. That is, such that
\[
\nabla(\Fil^iE)\subseteq (\Fil^{i-1}E)\tns[\cO_X] \Omega^1_X.
\]
Assume from now on that $E$ comes from a filtered convergent $F$-isocrystal $\cE$ on $\fX$. The de Rham cohomology of $X$ with coefficients in $E$ can be given the structure of a filtered $(\phi,N)$-module. Moreover, if $S$ is a finite set of points of $X$ and $U\dfn X\setminus S$, one can also give this structure to $\Hdr^1(U,E)$. The construction is detailed in~\cite[Subsection 2]{coleman2003hss}, and we recall it in the next two subsections. One needs to assume that the filtered $F$-isocrystal $\cE$ is \emphh{regular}, which is a condition on the characteristic polynomials of Frobenius acting on various crystalline cohomology groups. For the precise definition, see~\cite[Definition 2.3]{coleman2003hss}.

\subsection[\texorpdfstring{$\protect\Hdr^1(X,E)$}{Hdr1(X,E)}]{\texorpdfstring{$\protect\Hdr^1(X,E)$}{Hdr1(X,E)} as a filtered \texorpdfstring{$\protect (\phi,N)$}{phi-N}-module}
The (algebraic) de Rham cohomology of $X$ with coefficients in $E$, denoted by $\Hdr^*(X,E)$, is defined to be the hypercohomology of the complex of sheaves of $\cO_X$-modules:
\[
0\to E\tto{\nabla} E\tns\Omega^1_X\to 0.
\]
By rigid-analytic GAGA, this coincides with the rigid-analytic cohomology. We describe explicitly the space $\Hdr^1(X,E)$, using the admissible covering described above: an element $x\in\Hdr^1(X,E)$ can be represented by a $1$-hypercocycle
\[
\omega=\left(\{\omega_v\}_{v\in\fV(\fG)};\{f_e\}_{e\in\vec\fE(\fG)}\right),
\]
where the $\omega_v\in (E^\an\tns\Omega^1_{X^\an})(U_v)$, and $f_e\in E^\an(A_e)$ satisfy

\begin{align*}
\omega_{o(e)}|_{A_e}-\omega_{t(e)}|_{A_e}&=\nabla(f_e),\text{ and } f_{\bar e}=-f_e.
\end{align*}
Two such $1$-hypercocycles represent the same element $x\in\Hdr^1(X,E)$ if their difference is of the form
\[
\left(\{\nabla(f_v)\}_{v\in\fV(\fG)};\{f_{o(e)}|_{A_e}-f_{t(e)}|_{A_e}\}_{e\in\vec\fE(\fG)}\right),
\]
for some family $\{f_v\}_{v\in\fV(\fG)}$ with $f_v\in E^\an(U_v)$.

Assume from now on that the admissible opens $U_v$ and $A_e$ appearing in the covering are acyclic for coherent sheaf cohomology. Consider the maps induced by inclusion:
\begin{align*}
f&\colon \coprod_{v\in\fV(\fG)} U_v\to X,&g&\colon\coprod_{e\in\vec\fE(\fG)} A_e\to X.
\end{align*}
They give an exact sequence of sheaves on $X^\an$:
\[
0\to E^\an\to f_*f^*E^\an\to g_*g^*E^\an\to 0,
\]
which induces the Mayer-Vietoris long exact sequence:
\begin{align*}
0&\to \Hdr^0(X,E)\to\oplus_{v\in \fV(\fG)} \Hdr^0(U_v,E^\an)\to\oplus_{e\in \fE(\fG)} \Hdr^0(A_e,E^\an)\to\\
&\to\Hdr^1(X,E)\to\oplus_{v\in\fV(\fG)}\Hdr^1(U_v,E^\an)\to\oplus_{e\in\fE(\fG)}\Hdr^1(A_e,E^\an)\to\cdots
\end{align*}
We extract a short exact sequence
\begin{equation}
\label{eqn:basic}
0\to (H^0_\fE)^-/H^0_\fV\tto{\iota} \Hdr^1(X,E)\tto{\gamma} \ker \left(H^1_\fV\to H^1_\fE\right)\to 0,
\end{equation}
where
\begin{align*}
H^i_\fE&=\bigoplus_{e\in \vec\fE(\fG)} \Hdr^i(A_e,E^\an),&H^i_\fV&=\bigoplus_{v\in \fV(\fG)} \Hdr^i(U_v,E^\an),
\end{align*}
and the superscript $^-$ indicates the subspace of $H_\fE^0$ consisting of elements $\{f_e\}_e$ such that $f_{\bar e}=-f_e$ for all $e\in \vec\fE(\fG)$. The maps $\iota$ and $\gamma$ are given by the following recipe:
\begin{enumerate}
\item Let $\{f_e\}_{e\in\vec\fE(\fG)}$ with $f_e\in\Hdr^0(A_e,E^\an)$ satisfying $f_{\ol e}=-f_e$. Then $\iota$ sends the class of $\{f_e\}$ to the $1$-hypercocycle $(\{0\}_v;\{f_e\})$. Note that this is indeed a hypercocycle, since $\nabla f_e=0$.
\item Let $(\{\omega_v\}_v;\{f_e\}_e)$ be a $1$-hypercocycle representing a class $x\in\Hdr^1(X,E)$. Then $\gamma$ sends $x$ to the class of $\{\omega_v\}_v$ in $\bigoplus_v \Hdr^1(U_v,E^\an)$.
\end{enumerate}

In the following paragraphs we describe the structure as a filtered $(\phi,N)$-module of $\Hdr^1(X,E)$. The Hodge filtration is defined as
\[
\Fil^i\Hdr^1(X,E)\dfn \img\left(\bbH^1(X,\Fil^iE\tto{\nabla} \Fil^{i-1}E\tns\Omega^1_X)\to \bbH^1(X,E\tns \Omega^\bullet)\right).
\]
where the map $\bbH^1(X,\Fil^iE\tto{\nabla} \Fil^{i-1}E\tns\Omega^1_X)\to \bbH^1(X,E\tns \Omega^\bullet)$ is induced by functoriality from the inclusion of complexes
\[
\xymatrix{
\Fil^iE\ar@{^{(}->}[d]\ar[r]^-{\nabla}&\Fil^{i-1}E\tns[\cO_X]\Omega^1_X\ar@{^{(}->}[d]\ar[r]&\cdots\\
E\ar[r]^-{\nabla}& E\tns[\cO_X]\Omega^1_X\ar[r]&\cdots
}
\]
Note that this filtration coincides with the one induced from the Hodge to de Rham spectral sequence computing $\Hdr^*(X,E)$.

The Frobenius operator is defined by first splitting the exact sequence in Equation~\eqref{eqn:basic} and defining it in the outer terms. To define the splitting, we will use the Coleman integrals, so fix once and for all a branch of the $p$-adic logarithm. The map $\iota$ admits a retraction $P$ defined as follows: let $x\in \Hdr^1(X,E)$ be represented by the $1$-hypercocycle $(\{\omega_v\}_v;\{f_e\}_e)$. For any $v\in\fV(\fG)$, define $F_v$ to be a Coleman primitive of $\omega_v$, as introduced in Section~\ref{section:RigAn}. Then the map $P$ assigns to $x$ (the class of) the family $\{g_e\}_{e\in \vec\fE(\fG)}$, where
\[
g_e\dfn f_e-(F_{o(e)}|_{A_e}-F_{t(e)}|_{A_e}).
\]
Note that this map is well defined because the integrals $F_v$ are defined up to a rigid horizontal section of $E^\an|_{U_v}$.

There is an action of Frobenius on the left and right terms of the exact sequence~\eqref{eqn:basic}. That is, there are lattices inside the space $\Hdr^0(A_e,E^\an)$ and inside $\Hdr^1(U_v,E^\an)$, and respective actions of Frobenius. Concretely, if $e=(x,C_i,C_j)$, then $\Hcris^0(x,E^\an)$ is a $K_0$-lattice with a natural Frobenius. Also, $\Hdr^1(U_v,E^\an)$ has a natural lattice and action of Frobenius induced from the action on $\cE$.

Using the splitting $P$ we obtain a lattice inside $\Hdr^1(X,E)$, together with a Frobenius that will be called $\Phi$.

Lastly, we define the monodromy operator $N$. Associated to each open annulus $A_e$ such that $\left(E^\an|_{A_e}\right)$ has only constant horizontal sections, there is a natural annular residue map $\res_e=\res_{A_e}$:
\[
\res_e\colon \Hdr^1(A_e,E^\an)\to \Hdr^0(A_e,E^\an)\cong \left(E^\an|_{A_e}\right)^{\nabla=0}.
\]
For a fixed edge $e_0\in\vec\fE(\fG)$ there is a natural map $h_{e_0}\colon \Hdr^1(X,E)\to\Hdr^1(A_{e_0},E^\an)$ which sends $(\{\omega_v\}_v;\{f_e\}_e)$ to the class of $\omega_{o(e_0)}|_{A_e}$. Note that this coincides with the class of $\omega_{t(e_0)}|_{A_e}$. The monodromy operator $N$ on $\Hdr^1(X,E)$ is defined as:
\[
N\dfn \iota\circ \left(\oplus_e (\res_e\circ h_e)\right)\colon \Hdr^1(X,E)\to \Hdr^1(X,E).
\]
A simple computation shows that, as expected, $N$ satisfies:
\[
N\Phi=p\Phi N.
\]

\subsection{\texorpdfstring{$\protect\Hdr^1(U,E)$}{Hdr1(U,E)} as a filtered \texorpdfstring{$\protect (\phi,N)$}{phi-N}-module}
Let $S$ be a finite set of $K$-rational points on $X$ which are smooth (when considered as points on $\fX$), and which specialize to pairwise different smooth points on $C$.

Let $U=X\setminus S$. One can define, in a similar way as in the previous subsection, a structure of a filtered $(\phi,N)$-module on $\Hdr^1(U,E)$. The monodromy operator is defined as in the previous subsection. To define the Frobenius, one needs to work with logarithmic isocrystals. There is again an exact sequence
\[
0\to (H^0_\fE)^-/H^0_\fV\tto{\iota} \Hdr^1(U,E)\tto{\gamma} \ker \left(H^1_\fV\to H^1_\fE\right)\to 0,
\]
where this time
\begin{align*}
H^i_\fE&=\bigoplus_{e\in \vec\fE(\fG)} \Hdr^i(A_e,E^\an),&H^i_\fV&=\bigoplus_{v\in \fV(\fG)} \Hdr^i(U_v\setminus S,E^\an).
\end{align*}

The left-most term is the same as before, because the zeroth cohomology does not change by removing a finite set of points. So to define the Frobenius on $\Hdr^1(U,E)$ one has to define it on the right-most term. This is done in \cite[Subsection 5]{coleman2003hss}, where the de Rham cohomology $\Hdr^1(U_v\setminus S,E^\an)$ is described in terms of the log-crystalline cohomology with coefficients in $j^*E$ of the component $C_i=v$ of $C$, where $j$ is the canonical morphism of formal log-schemes $j\colon (\widehat\fX,\text{log structure})\to (\widehat\fX,\text{trivial})$.

The Gysin sequence
\[
\xymatrix{
0\ar[r]& \Hdr^1(X,E)\ar[r]& \Hdr^1(U,E)\ar[r]^-{\oplus \res_x}&\bigoplus_{x\in S}\cE_x[1]
}
\]
becomes in this way an exact sequence of filtered $(\phi,N)$-modules.

Let $f\colon\fY\to\fX$ be a smooth proper morphism, and let $Y$ be the generic fiber of $\fY$. The relative de Rham cohomology
\[
\cHdr^q(Y/X)\dfn R^qf_*\cO_{\widehat\fY/K}
\]
can be given the structure of a filtered convergent F-isocrystal, as in Example~\ref{ex:isocrystals}(\ref{ex:derhamisocrystal}). Consider the $G_{\QQ}$-representation
\[
\Het^1(\overline X,R^qf_*\QQ_p),
\]
and the filtered $(\phi,N)$-module
\[
\Hdr^1\left(X,\cHdr^q(Y/X)\right)
\]
defined above. The following result relates these two objects. Its proof can be found in~{\cite[Theorem~7.5]{coleman2003hss}}.

\begin{thm}[(Faltings, Coleman-Iovita)]
Using the previous notations:
\label{thm:is36}
\begin{enumerate}
\item The representation $\Het^1(\overline X,R^qf_*\QQ_p)$ is semistable, and there is a canonical isomorphism of filtered $(\phi,N)$-modules
\[
\Dst\left(\Het^1(\overline X,R^qf_*\QQ_p)\right)\cong \Hdr^1\left(X,\cHdr^q(Y/X)\right).
\]
\item More generally, let $S$ be a finite set of smooth sections of $f\colon \fX\to \Spec(\cO_K)$, which specialize to pairwise different (smooth) points on $C$, and let $U=X\setminus S$, $\overline U=U\tns[K]\overline K$, and $Y_{\overline x}$ be the geometric fiber of $f\colon Y\to X$ over $x\in S$. Then there is an exact sequence of semistable Galois representations
\[
0\to \Het^1(\overline X,R^qf_*\QQ_p)\to\Het^1(\overline U,R^qf_*\QQ_p)\to \bigoplus_{s\in S} \Het^q(Y_{\overline x},\QQ_p(-1)),
\]
which becomes isomorphic to the sequence
\[
0\to \Hdr^1(X,E)\to \Hdr^1(U,E)\to \bigoplus_{x\in S} \cE_x[1]
\]
after applying the functor $\Dst$ (and setting $\cE=\cHdr^q(Y/X)$).
\end{enumerate}
\end{thm}

\section{The cohomology of $X_\Gamma$, and pairings}
\label{section:Cohomology}
\label{section:CohomologyXGamma}
\subsection[\texorpdfstring{$\protect\Hdr^1(X_\Gamma,E(V))$}{Hdr1(X,E(V))}]{The filtered \texorpdfstring{$\protect(\phi,N)$}{phi-N}-module \texorpdfstring{$\protect\Hdr^1(X_\Gamma,E(V))$}{Hdr1(X,E(V))}}
\label{subsection:cohomologyXGamma}
We want to specialize the constructions made in the previous subsections to the situation in our work. We will assume that the curve $X$ is a certain quotient of the $p$-adic upper-half plane, and we will restrict also the class of filtered convergent $F$-isocrystals that we consider. Let $V$ be an object of $\Rep_{\QQ_p}(\gl_2\times\gl_2)$. In \cite[Subsection~4]{iovita2003dpa} the authors associate to $V$ a filtered convergent $F$-isocrystal on the canonical formal $\ZZ_p^\ur$-model of the upper-half plane $\widehat\cH$, which is denoted $\cE(V)$. Also, for every $\QQ_{p^2}$-rational point $\Psi\in \Hom(\QQ_{p^2},M_2(\QQ_p))$ of $\widehat\cH$ they compute the stalk $\cE(V)_\Psi$ as a filtered $(\phi,N)$-module $V_\Psi\in \MF_{\QQ_p^\text{ur}}^{(\phi,N)}$. The assignment $V\mapsto \cE(V)$ is an exact tensor functor.

The previous construction can be descended to give isocrystals on Mumford curves: if $\Gamma$ is a discrete cocompact subgroup of $\SL_2(\QQ_p)$, let $X_\Gamma$ be the associated Mumford curve over $\QQ_p^\text{ur}$, so that $X_\Gamma^{\text{an}}=\Gamma\backslash\cH_p$. Denote the new filtered isocrystal on $X_\Gamma$ by the same symbol $\cE(V)$ as well.

Let $E(V)$ be the coherent locally free $\cO_{X_\Gamma}$-module with connection and filtration corresponding to $\cE(V)$, so that $\cE(V)=E(V)^\an$.

In \cite{iovita2003dpa} the authors give a concrete description of the filtered $(\phi,N)$-module
$\Hdr^1(X_\Gamma,E(V))$ and, if $U\subseteq X_\Gamma$ is an open subscheme as before, also of the filtered $(\phi,N)$-module $\Hdr^1(U,E(V))$. This is possible because both the curve $X_\Gamma$ and the coefficients $E(V)$ are known explicitly. We will assume that $\Gamma$ is torsion free. We can reduce to this situation as follows: choose $\Gamma'\subset \Gamma$ a free normal subgroup of finite index. The group $\Gamma/\Gamma'$ acts on the filtered $(\phi,N)$-modules $\Hdr^1(X_\Gamma,E(V))$ and $\Hdr^1(U,E(V))$ as automorphisms preserving the operators and the filtration. Hence it induces a structure of filtered $(\phi,N)$-module on
\[
\Hdr^1(X_\Gamma,E(V))\dfn \Hdr^1(X_{\Gamma'},E(V))^{\Gamma/\Gamma'},
\]
and similarly for $\Hdr^1(U,E(V))$.

\label{subsection:hdrXEV}

The fact that $\cH_p$ is a Stein space in the rigid-analytic sense  allows for the computation of $\Hdr^1(X_\Gamma,E(V))$ as group hyper-cohomology, via the Leray spectral sequence. More precisely, the $K$-vector space $\Hdr^1(X_\Gamma,E(V))$ can be computed as the first group hyper-cohomology:
\[
\Hdr^1(X,E(V))=\bbH^1(\Gamma,\Omega^\bullet\tns V),
\]
where $\Omega^\bullet$ is the de Rham complex
\[
\Omega^\bullet:\qquad 0\to \cO_{\cH_p}(\cH_p)\to \Omega^1_{\cH_p}(\cH_p)\to 0.
\]
Concretely, the elements in $\Hdr^1(X,E(V))$ are represented by pairs $(\omega,f_\gamma)$, where $\omega$ belongs to $\Omega^1(\cH_p)\tns V$ and $f_\gamma$ is a $\cO_{\cH_p}(\cH_p)\tns V$-valued $1$-cocycle for $\Gamma$. They are required to satisfy the relation
\[
\gamma\omega-\omega=df_\gamma,\quad\text{for all } \gamma\in\Gamma.
\]

Let $\Gamma$ be a cocompact subgroup of $\PSL_2(\QQ_p)$, and let $M$ be a $\CC_p[\Gamma]$-module. An \emph{$M$-valued $0$-cocycle} (resp.~$1$-cocycle) on $\cT$ is an $M$-valued function $c$ on $\fV(\cT)$ (resp.~on $\fE(\cT)$, such that $c(\overline e)=-c(e)$). The $\CC_p$-vector space of $M$-valued $0$-cocycles (resp.~$1$-cocyles) is written $C^0(M)$ (resp.~$C^1(M)$). An $M$-valued $0$-cocycle $c$ is called \emph{harmonic} if it satisfies, for all $v\in\fV(\cT)$,
\[
\sum_{e\in\fE(\cT),o(e)=v} c(o(e))-c(t(e))=0.
\]
The $\CC_p$-vector space of $M$-valued harmonic $0$-cocycles is written $\Char^0(M)$. An $M$-valued $1$-cocycle $c$ is called \emph{harmonic} if it satisfies
\[
\sum_{o(e)=v} c(e)=0
\]
for all $v\in\fV(\cT)$. The $\CC_p$-vector space of $M$-valued harmonic $1$-cocycles is written $\Char^1(M)$.

The group $\Gamma$ acts on the left on $\Char^i(M)$ by
\[
\gamma\cdot c\dfn \gamma\circ c\circ \gamma^{-1},\quad \gamma\in\Gamma,c\in\Char^i(M).
\]

Let $\cP_n$ be the $n+1$-dimensional $\QQ_p$-vector space of polynomials of degree at most $n$ with coefficients in $\QQ_p$. The group $\GL_2(\QQ_p)$ acts on the right on $\cP_n$ by
\[
P(x)\cdot \beta\dfn (cx+d)^nP\left(\frac{ax+b}{cx+d}\right),
\]
if $\beta=\smtx abcd$. In this way its $\QQ_p$-linear dual $V_n\dfn\cP_n^\vee$ is endowed with a left action of $\GL_2(\QQ_p)$.

\begin{df}
A \emph{harmonic cocycle of weight $n+2$ on $\cT$} is a $V_{n}$-valued harmonic cocycle.
\end{df}

Define now $\cU$ as the subspace of $M_2(\QQ_p)$ given by matrices of trace $0$. They have a right action of $\GL_2(\QQ_p)$ given by
\[
u\cdot \beta\dfn \overline{\beta} u \beta,
\]
where $\overline{\beta}$ is the matrix such that $\overline{\beta}\beta=\det(\beta)$.

There is a map $\Phi\colon \cU\to \cP_2$ intertwining the action $\GL_2(\QQ_p)$, given by
\begin{equation}
\label{eq:defpoly}
u\mapsto P_u(x)\dfn \tr\left(u\smat{x&-x^2\\1&-x}\right)=\tr\left(u\smat{x\\1}\smat{1&-x}\right)=\smat{1&-x}u\smat{x\\1}.
\end{equation}

\begin{lemma}
The map $\Phi$ induces an isomorphism of right $\GL_2(\QQ_p)$-modules.
\end{lemma}

On $\cU$ there is a pairing defined by
\[
\langle u,v\rangle\dfn -\tr(u\overline{v}).
\]
This induces a pairing on $\cP_2$ by transport of structure, and on the dual $V_2$ of $\cP_2$ by canonically identifying $\cP_2$ with $V_2$ using the pairing itself. Unwinding the definitions, we can prove the following formula:

\begin{lemma}
Take as basis for $V_2$ the linear forms $\{\omega_i\}_{0\leq i\leq 2}$, dual to the basis $\{1,x,x^2\}$ of $\cP_2$. Then the pairing $\langle\cdot,\cdot\rangle$ on $V_2$ is given by:
\[
\langle a\omega_0+b\omega_1+c\omega_2,a'\omega_0+b'\omega_1+c'\omega_2\rangle=2bb'-a'c-ac'.
\]
\end{lemma}

The pairing $\langle\cdot,\cdot\rangle$ induces a perfect symmetric pairing on $\Sym^n V_2=V_n$ given by the formula:
\[
\langle v_1\cdots v_n,v'_1\cdots v'_n\rangle\dfn \frac{1}{n!}\sum_{\sigma\in \fS_n} \langle v_1,v'_{\sigma(1)}\rangle\cdots\langle v_n,v'_{\sigma(n)}\rangle.
\]

We define a map $\epsilon\colon C^1(V_{\QQ_p^\ur})^\Gamma\to H^1(\Gamma,V_{\QQ_p^\ur})$, as follows: given a $1$-cocycle $f\in C^1(V_{\QQ_p^\ur})^\Gamma$, the element $\epsilon(f)$ is defined as the cohomology class of the $1$-cocyle
\[
\gamma\mapsto \gamma F(\gamma^{-1}(\star)),
\]
where $\star\in\fV(\cT)$ is a choice of a vertex of $\cT$, and $F\in C^0(V_{\QQ_p^\ur})$ satisfies $\partial F=f$ and $F(\star)=0$. One easily checks that this definition does not depend on the choice of the vertex $\star$.

\begin{prop}
The map $\epsilon$ induces an isomorphism:
\[
\epsilon\colon C^1(V_{\QQ_p^\ur})^\Gamma/C^0(V_{\QQ_p^\ur})^\Gamma\to H^1(\Gamma,V_{\QQ_p^\ur}).
\]
\end{prop}
\begin{proof}
  Consider the short exact sequence
\[
0\to V_{\QQ_p^\ur}\to C^0(V_{\QQ_p^\ur})\tto{\partial} C^1(V_{\QQ_p^\ur})\to 0.
\]
The map $\epsilon$ is the map induced from the connecting homomorphism $\delta$ in the $\Gamma$-cohomology long exact sequence:
\[
0\to C^0(V_{\QQ_p^\ur})^\Gamma\tto{\partial} C^1(V_{\QQ_p^\ur})^\Gamma\tto{\delta} H^1(\Gamma,V_{\QQ_p^\ur})\to H^1(\Gamma,C^0(V_{\QQ_p^\ur}))\to\cdots
\]
Since $\Gamma$ is torsion-free, $H^1(\Gamma,C^0(V_{\QQ_p^\ur}))=0$, and the result follows.
\end{proof}

Note that there is a canonical isomorphism
\[
\Char^1(V_{\QQ_p^\ur})^\Gamma\cong C^1(V_{\QQ_p^\ur})^\Gamma/C^0(V_{\QQ_p^\ur})^\Gamma,
\]
and we will identify these two spaces from now on. Let $\gamma\mapsto f_\gamma$ be a $1$-cocycle on $\Gamma$. The group $C^0(V_{\QQ_p^\ur})$ is $\Gamma$-acyclic, so that there is a $0$-harmonic cocycle $F\in C^0(V_{\QQ_p^\ur})$ satisfying, for all $\gamma\in\Gamma$,
\[
j(f_\gamma)=\gamma F - F.
\]
Consider then $\partial(F)\in C^1(V_{\QQ_p^\ur})$. It is fixed by $\Gamma$, since:
\[
(\gamma\partial(F))-\partial(F)=\partial(\gamma F-F)=\partial(j(f_\gamma))=0.
\]
Then the class of the $1$-hypercocycle given by $(\{0\}_v;\{F(o(e))-F(t(e))\}_{e})$ is an element of $\Hdr^1(X_\Gamma,E(V))$. We thus obtain an injection:
\[
\iota\colon H^1(\Gamma,V_{\QQ_p^\ur})\to \Hdr^1(X_\Gamma,E(V)).
\]

Next we construct a map  $I\colon \Hdr^1(X_\Gamma,E(V))\to \Char^1(V_{\QQ_p^\ur})^\Gamma$ which is due to Schneider. It is called ``Schneider integration'' in \cite{iovita2003dpa}, \cite{deshalit1989eca} and \cite{deshalit:civ}. Denote by $\Omega^{\text{II}}_{X^\an_\Gamma}$ the space of meromorphic differentials of the second kind on $X^\an_\Gamma$. There is a map $\Omega^{\text{II}}_{X^\an_\Gamma}\tns[\cO_{X_\Gamma}] V\to C^1(V_{\QQ_p^\ur})$:
\[
\omega\mapsto \left(e\mapsto \res_e(\omega)\right),
\]
and this map induces $I$. The residue theorem implies that the image of $I$ lies in $\Char^1(V_{\QQ_p^\ur})^\Gamma$.
\begin{lemma}[de Shalit]
\label{lemma:deshalit}
Suppose that $\Gamma$ is arithmetic. Then the sequence
\begin{equation}
\label{eqn:deshalit}
\xymatrix@R3pt{
0\ar[r]&H^1(\Gamma,V_{\QQ_p^\ur})\ar[r]^-{\iota}&\Hdr^1(X,E(V))\ar[r]^-{I}&\Char^1(V_{\QQ_p^\ur})^\Gamma\ar[r]&0
}
\end{equation}
is exact.
\end{lemma}

There is a retraction $P$ of $\iota$, given by Coleman integration. This assigns to $(\omega,f_\gamma)$ the $1$-cocycle
\[
\gamma\mapsto f_\gamma+\gamma F_\omega-F_\omega,
\]
where $F_\omega$ is a Coleman primitive for $\omega$ as in Theorem~\ref{thm:coleman-integral}. Note that the $V_{\QQ_p^\ur}$-valued function $\gamma F_\omega-F_\omega$ is constant, so that we can think of it as a well-defined element of $V_{\QQ_p^\ur}$.

The splitting $P$ thus defines actions of Frobenius on the left and right terms of the exact sequence~\eqref{eqn:deshalit}, as follows: there is a natural action $\phi_1$ of Frobenius on $H^1(\Gamma,V_{\QQ_p^\ur})$. Define an action $\phi_2$ on $\Char^1(V_{\QQ_p^\ur})^\Gamma$ as $\phi_2\dfn p(\epsilon^{-1}\circ\phi_1\circ\epsilon)$, so that the following equality holds:
\[
\epsilon\phi_2=p\phi_1\epsilon.
\]

We have now all the maps needed in the definition of the Frobenius and monodromy operators. Define first $N$ to be the composition $\iota\circ (-\epsilon)\circ I$. Since $I\circ\iota=0$ it follows that $N^2=0$. Actually, Lemma~\ref{lemma:deshalit} implies that $\ker N=\img N$. Let $T$ be the right-inverse to $I$ corresponding to $P$:
\[
\xymatrix{
0\ar[r]& H^1(\Gamma,V_{\QQ_p^\ur})\ar[r]^-{\iota}\ar@(ur,ul)[]_{\phi_1}& \Hdr^1(X_\Gamma,E(V))\ar[r]^-{I}\ar@/_1.5pc/@{-->}[l]_{P}& C^1(V_{\QQ_p^\ur})^\Gamma/C^0(V_{\QQ_p^\ur})^\Gamma\ar[r]\ar@/_1.5pc/@{-->}[l]_{T}\ar@(ur,ul)[]_{\phi_2}& 0.
}
\]
Define the Frobenius operator $\Phi$ on $\Hdr^1(X_\Gamma,E(V))$ as:
\[
\Phi(\omega)\dfn \iota\phi_1(P\omega)+ T(\phi_2(I\omega)).
\]
It can easily be checked that this definition satisfies $N\Phi=p\Phi N$, and that $\Phi$ is the unique such action which is compatible with the maps $P$ and $\iota$.

Let $S$ be a finite set of points of $X_\Gamma$, and let  $U=X_\Gamma\setminus S$ be the open subscheme obtained by removing the points in $S$. The space $\Hdr^1(U,E(V))$ is identified with the space of $V$-valued differential forms on $X_\Gamma^\an$ which are of the second kind when restricted to $U$. The monodromy is defined in the same way as before. The Frobenius is defined so that the Gysin sequence
\begin{equation}
\label{eqn:gysin-seq}
\xymatrix{
0\ar[r]& \Hdr^1(X_\Gamma,E(V))\ar[r]& \Hdr^1(U,E(V))\ar[rr]^{\oplus_{x\in S}\res_x}&& \bigoplus_{x\in S} V_{\Psi_x}[1]
}
\end{equation}
is a sequence of $(\phi,N)$-modules, and such that $P$ is compatible with the Frobenii.

\subsection{A special case of interest}
\label{sec:specialcase}
In \cite{iovita2003dpa} the authors apply the previous constructions to a filtered isocrystal on $\cH_p$ denoted by $\cE(M_2)$. It is shown in \cite[Lemma~5.10]{coleman2003hss} that $\cE(M_2)$ is regular, and therefore one can define a structure of a filtered $(\phi,N)$-module on its cohomology groups. Here we make the construction explicit. Consider first the $\QQ_p$-vector space of $2\times 2$ matrices $M_2$. Define two commuting left actions of $\GL_2$ on $M_2$ by:\begin{align*}
\rho_1(A)(B)&\dfn AB & \rho_2(A)(B)&\dfn B\overline A,
\end{align*}
for $A\in\GL_2$ and $B\in M_2$. The matrix $\overline A$ is such that $A\overline A=\det A$. This gives a representation:
\[
(M_2,\rho_1,\rho_2)\in \Rep_{\QQ_p}(\GL_2\times\GL_2).
\]
The isocrystal $\cE(M_2)$ is constructed as an isocrystal on the canonical formal model $\widehat\cH$ over $\ZZ_p^{\ur}$ of $\cH_p$. We are also interested in its fibers  over $\QQ_{p^2}$-rational points $\Psi\in \Hom(\QQ_{p^2},M_2(\QQ_p))$. Let $\cO_{\widehat\cH}$ be the isocrystal attached to $\widehat \cH$. As an isocrystal $\cE(M_2)$ is:
\[
M_2(\QQ_p)\tns[\QQ_p] \cO_{\widehat\cH}.
\]
We need to define the Frobenius and filtration. First, let $\phi$ be the action on $M_2(\QQ_p)$:
\[
\smtx abcd \mapsto  \smtx abcd \smtx 0{-p}{-1}0 =\smtx{-b}{-pa}{-d}{-pc}.
\]
On $\cO_{\widehat\cH}$ there is a Frobenius action $\Phi_{\cO_{\widehat\cH}}$ as well. We define the Frobenius on the tensor product through these two actions.

The filtration on $M_2$ is given in degree $1$ by:
\[
F^1 M_2=\left\{\smtx{zf(z)}{zg(z)}{f(z)}{g(z)}\stt f,g\in\cO_{\cH_p}\right\}.
\]

We end by describing the stalks of $\cE(M_2)$. For each point $\Psi\in \Hom(\QQ_{p^2},M_2(\QQ_p))$, the stalk $\cE(M_2)_\Psi$ is, as a $\QQ^\ur_p$-vector space, the space $M_2(\QQ^\ur_p)$ of $2\times 2$ matrices. The Frobenius acts by $\phi$ on $M_2$ and by $\sigma$ on $\QQ_p^\ur$. That is:
\[
\Phi\left(\smtx abcd\right)=\smtx{-\sigma(b)}{-p\sigma(a)}{-\sigma(d)}{-p\sigma(c)}.
\]

To describe the filtration on $M_2(\QQ_p^\ur))$, we first consider for each $j$ the subspace $V_j$:
\[
V_j\dfn \{A\in M_2(\QQ_p^\ur)\stt \Psi(x)A=x^j\sigma(x)^{1-j}A, \forall x\in\QQ_{p^2}\}.
\]
The filtration on $M_2(\QQ_p^\ur)$ is:
\[
\Fil^i_{\Psi}M_2(\QQ_p^\ur)\dfn \bigoplus_{j\geq i} V_j.
\]
This is an exhaustive and separated filtration on $M_2(\QQ_p^\ur)$. The monodromy is trivial.

\subsection{A pairing on \texorpdfstring{$\protect\Hdr^1(X_\Gamma,E(V))$}{Hdr1(X,E(V))}}
\label{subsection:pairing}
Let $V$ be a finite-dimensional representation of $\Gamma$ over $K$ endowed with a $\Gamma$-invariant perfect pairing $\langle\cdot,\cdot,\rangle_V$. We first describe a pairing $\langle\cdot,\cdot\rangle_\Gamma$:
\[
\langle\cdot,\cdot\rangle_\Gamma\colon \Char^1(V)^\Gamma\tns H^1(\Gamma,V)\to K,
\]
given as follows: choose a free subgroup $\Gamma'\subset\Gamma$ of finite index, and let $\fF$ be a good fundamental domain for $\Gamma'$ as in~\cite[Section~2.5]{deshalit1989eca}. Let $b_1,\ldots,b_g,c_1,\ldots,c_g$ be the free edges for $\fF$. For and $f\in\Char^1(V)^\Gamma$ and $[z]\in H^1(\Gamma,V)$, the pairing is given by the formula:
\[
\langle [z],f\rangle_{\Gamma}=\frac{1}{[\Gamma\colon\Gamma']}\sum_{i=1}^g \langle z(\gamma_i),f(c_i)\rangle_{V}.
\]

The previous pairing induces a pairing on $\Hdr^1(X_\Gamma,E(V))$, and we are interested in a formula for it, which we now proceed to describe. Let $x,y\in\Hdr^1(X_\Gamma,E(V))$. E.~de~Shalit computed first a formula for this pairing in~\cite{deshalit1989eca} and~\cite{deshalitfcp}, and Iovita-Spie\ss{} proved it in a more conceptual way which allowed for a generalization, in \cite{iovita2003dpa}. They obtained the equality:
\begin{equation}
\label{eqn:cupproduct}
\langle x,y\rangle_{X_\Gamma}=\langle P(x),I(y)\rangle_\Gamma-\langle I(x),P(y)\rangle_\Gamma.
\end{equation}

\subsection{Pairings between \texorpdfstring{$\Hdr^1(U,E(V))$}{Hdr1(U,E(V))} and \texorpdfstring{$\Hdrc^1(U,E(V))$}{Hdrc1(U,E(V))}}
In this subsection we make explicit some of the constructions carried out in~\cite[Appendix]{iovita2003dpa}.

Let $U=X\setminus\{x\}$, where $x$ is a  closed point of $X$ defined over the base field $K$. Write $j\colon U\to X=X_\Gamma$ for the canonical inclusion. Let $z$ be a lift of $x$ to $\cH_p(K)$, taken inside a good fundamental domain $\cF$. We assume that the stabilizer of $z$ under the action of $\Gamma$ is trivial.
Let $\Ind^\Gamma(V)$ be the $\Gamma$-representation given by $\Maps(\Gamma,V)$, with $\Gamma$-action:
\[
(\gamma\cdot f)(\tau)\dfn \gamma f (\gamma^{-1}\tau).
\]
Let $\ad\colon V\to\Ind^\Gamma(V)$ be defined as the constant map: $\ad(v)(\tau)\dfn v$.
Consider the complex $\cK^\bullet(V)$, concentrated on degrees $0$ and $1$, defined as:
\[
\cK^\bullet(V):\qquad V\tto{\ad} \Ind^\Gamma(V).
\]
Consider also the complex $C^\bullet(V)$ defined as follows:
\[
C^\bullet(V):\qquad \cO_{\cH_p}(\cH_p)\tns V\tto{(d,\ev_z)} \Omega^1(\cH_p)\tns V\oplus \Ind^\Gamma(V).
\]

\begin{df}
The \emphh{cohomology with compact support} on $U$ with coefficients in $E(V)$ is the hypercohomology group:
\[
\Hdrc^1(U,E(V))\cong \bbH^1(\Gamma,C^\bullet(V)).
\]
\end{df}

\begin{fact}

The inclusion $j$ induces natural maps:
\[
j_*\colon \Hdrc^1(U,E(V))\to\Hdr^1(X,E(V))
\]
and
\[
j^*\colon\Hdr^1(X,E(V))\to\Hdr^1(U,E(V)).
\]
There is a pairing $\langle\cdot,\cdot\rangle_U$, on:
\[
\Hdrc^1(U,E(V))\times\Hdr^1(U,E(V))\to K,
\]
induced from the cup-product. It satisfies:
\[
\langle j_* y_1,y_2\rangle_X=\langle y_1,j^*y_2\rangle_U.
\]
\end{fact}

The exact triangle:
\[
\cK^\bullet(V)\to C^\bullet(V)\to \Char^1(V)[-1]\to\cK^\bullet(V)[1]
\]
induces a short exact sequence:
\[
0\to \bbH^1(\Gamma,\cK^\bullet(V))\tto{\iota_{U,c}} \Hdrc^1(U,E(V))\tto{I_{U,c}} \Char^1(V)^\Gamma\to 0.
\]
The splitting $P_{U,c}$ is defined as follows. Fix a branch of the $p$-adic logarithm. Let $\cF(V)$ be the subspace of those $V$-valued locally-analytic functions on $\cH_p$ which are primitives of elements of $\Omega^1(\cH_p)\tns V$. There is an exact sequence:
\[
0\to V\to \cF(V)\tto{d} \Omega^1(\cH_p)\tns V\to 0,
\]
and one immediately checks that this implies that the complex
\[
\cF(V)\to \Omega^1(\cH_p)\tns V\oplus \Ind^\Gamma(V)
\]
is quasi-isomorphic to $\cK^\bullet$.
\[
P_{U,c}\colon \Hdrc^1(U,E(V))=\bbH^1(\Gamma,C^\bullet)\to \bbH^1(\Gamma,\cK^\bullet).
\]

We define now a $\Gamma$-module $C_U(V)$. There is a surjective map
\[
\delta\colon C^1(V)\to C^0(V),
\]
defined by:
\[
\delta(f)(v)\dfn\sum_{o(e)=v} f(e).
\]
Let $v_0\dfn \red(z)$, and define $\chi\colon\Ind^\Gamma(V)\to C^0(V)$ by:
\[
\chi(f)(v)\dfn
\begin{cases}
  f(\gamma)&\If v=\gamma v_0,\text{ for some }\gamma\in\Gamma,\\
0&\Else .
\end{cases}
\]
The $\Gamma$-module $C_U(\Gamma)$ is defined to be the kernel of the map:
\[
C^1(V)\bigoplus \Ind^\Gamma(V)\to C^0(V),
\]
mapping $(f,g)\mapsto \delta(f)-\chi(g)$. The map $I_U\colon \Hdr^1(U,E(V))\to C_U(V)^\Gamma$ is naturally induced from the map $\tilde I_U\colon \Omega^1(\cH_p)(\log(|z|))\tns V\to C_U(V)$, defined by:
\[
\tilde I_U(\omega)(e,\gamma)\dfn(\res_e(\omega),\res_{\gamma(z)}(\omega)).
\]
Also, the map $\iota_U$ is induced from the natural inclusion
\[
V\to\cO_{\cH_p}(\cH_p)(\log(|z|))\tns V.
\]
Finally, we define a splitting $P_U$ is defined by the same formula as the one defining $P$. We end this section by recalling the explicit description of the pairing
\[
\langle\cdot,\cdot\rangle_{\Gamma,U}\colon H^1(\Gamma,\cK^\bullet(V))\tns C_U(V)^\Gamma\to K.
\]

\begin{prop}[Iovita-Spie\ss{}]
\label{prop:formula-for-pairing}
  Let $x\in H^1(\Gamma,\cK^\bullet(V))$ be represented by  $(\zeta,f)$, such that
\[
\ad\circ \zeta=\partial(f),
\]
with $\zeta\in Z^1(\Gamma,V)$ a one-cocycle and $f\in \Ind^\Gamma(V)$ satisfying
\[
(\partial f)(\gamma)=\gamma f-f.
\]
Let $(g,g')\in C^1(V)\oplus\Ind^\Gamma(V)$ be an element in $C_U(V)^\Gamma$, so that $\delta(g)=\chi(g')$. Choose a free subgroup $\Gamma'\subset\Gamma$ of finite index, and let $\fF$ be a good fundamental domain for $\Gamma'$  as in~\cite[Section~2.5]{deshalit1989eca}. Let $b_1,\ldots,b_g,c_1,\ldots,c_g$ be the free edges for $\fF$. Then:
\[
\langle [(\zeta,f)],(g,g')\rangle_{\Gamma,U} =\frac{1}{[\Gamma\colon \Gamma']} \sum_{i=1}^g\langle \zeta(\gamma_i),g(c_i)\rangle+\langle f(1),g'(1)\rangle.
\]
\end{prop}
\begin{proof}
  See \cite[Appendix]{iovita2003dpa}.
\end{proof}

We have constructed a commutative diagram with exact split rows:
\begin{equation}
\label{eq:splittings}
\xymatrix{
0\ar[r]&\bbH^1(\Gamma,\cK^\bullet(V))\ar[r]^{\iota_{U,c}}\ar[d]^{q_*}&\Hdrc^1(U,E(V))\ar[r]^{I_{U,c}}\ar@/_1.5pc/@{-->}[l]_{P_{U,c}}\ar[d]^{j_*}&\Char^1(V)^\Gamma\ar[r]\ar@{=}[d]&0\\
0\ar[r]&H^1(\Gamma,V)\ar[r]^{\iota}\ar@{=}[d]&\Hdr^1(X,E(V))\ar[r]^{I}\ar@/_1.5pc/@{-->}[l]_{P}\ar[d]^{j^*}&\Char^1(V)^\Gamma\ar[r]\ar@{-->}[d]&0\\
0\ar[r]&H^1(\Gamma,V)\ar[r]^{\iota_{U}}&\Hdr^1(U,E(V))\ar[r]^{I_{U}}\ar@/_1.5pc/@{-->}[l]_{P_{U}}&C_U(V)^\Gamma\ar[r]&0.
}
\end{equation}
Here the bent arrows mean splittings of the corresponding maps, and the vertical dotted arrow means the natural induced map on the quotient.

\section[The anti-cyclotomic \texorpdfstring{$p$}{p}-adic \texorpdfstring{$L$}{L}-function]{The anti-cyclotomic \texorpdfstring{$p$}{p}-adic \texorpdfstring{$L$}{L}-function}
\label{section:LFunction}

In~\cite{bertolini2002tse} the authors define the anti-cyclotomic $p$-adic $L$-function. Assume from now on that $p$ is inert in $K$, and fix an isomorphism $\iota\colon B_p\to M_2(\QQ_p)$.

\subsection{Distributions associated to modular forms}
\label{sec:distribution-modforms}
Let $f$ be a rigid analytic modular form of weight $n+2$ on $\Gamma$, and denote by $\cA_{n}$ the set of $\CC_p$-valued functions on $\PP^1(\QQ_p)$ which are locally analytic except for a pole of order at most $n$ at $\infty$. Note that the subspace $\cP_{n}$ of polynomials of degree at most $n$ is dense in $\cA_n$. In~\cite{teitelbaum1990voa} Teitelbaum associates to $f$ a distribution $\mu_f$ on $\cA_n$, in such a way that it vanishes on polynomials of degree at most $n$. This is done by extending to $\cA_n$ the measure defined by:
 \[
 \mu_f(P\cdot \chi_{U(e)})\dfn\int_{U(e)} P(x)d\mu_f(x)\dfn \res_e(f(z)P(z)dz),
 \]
where the polynomial $P$ belongs to $\cP_{n}$, and $e$ is an end in $\fE_\infty(\cT)$ such that $\infty\notin U(e)$.

This distribution extends uniquely to $\cA_{n}$, and the $p$-adic residue formula
gives:

\begin{lemma}
\label{lemma:technical}
If $P\in \cP_{n}$, then
\[
\int_{\PP^1(\QQ_p)} P(x)d\mu_f(x)=0.
\]
\end{lemma}
\begin{proof}
  Write first
\[
\PP^1(\QQ_p)=\coprod_{i=0}^p U(e_i),
\]
where $e_0,\ldots,e_{p}$ are the $p+1$ edges leaving the origin vertex $v_0$. Then:
\[
\int_{\PP^1(\QQ_p)} P(x)d\mu_f(x)=\left(\sum_{i=0}^p c_f(e_i)\right)(P(X))=0,
\]
because $c_f$ is a harmonic cocycle.
\end{proof}

The group $\GL_2(\QQ_p)$ acts also on $\cA_{n}$ with weight $n$, by the rule:
\[
(\varphi * \beta)(x)\dfn (cx+d)^{n}\varphi(\beta \cdot x),\quad \varphi\in \cA_{n}\text{, and }\beta\in \PGL_2(\QQ_p).
\]

One can also recover a modular form $f$ from its associated distribution:
\begin{prop}[Teitelbaum]
\label{prop:Teitelbaum}
Let $f$ be a rigid analytic modular form of weight $n+2$ on $\Gamma$, and let $\mu_f$ be the associated distribution on $\PP^1(\QQ_p)$. Then
\[
f(z)=\int_{\PP^1(\QQ_p)}\frac{1}{z-t}d\mu_f(t).
\]
\end{prop}
\begin{proof}
See~\cite[Theorem 3]{teitelbaum1990voa}.
\end{proof}

\subsection{Construction of the \texorpdfstring{$p$}{p}-adic \texorpdfstring{$L$}{L}-function}
For the rest of the paper, assume that $K$ is a number field satisfying:
\begin{enumerate}
\item all primes dividing $pN^-$ are inert in $K$, and
\item all primes dividing $N^+$ are split in $K$.
\end{enumerate}
In particular, note that we require the discriminant of $K$ to be coprime to $N=pN^-N^+$. An embedding $\Psi\colon K\to B$ is called \emphh{optimal} if $\Psi(K)\cap R=\Psi(\cO)$, so that $\Psi$ induces an embedding of $\cO$ into $R$. What will be called the partial $p$-adic $L$-function depends on a pair $(\Psi,\star)$, of an optimal embedding $\Psi\colon K\to B$ and a base-point $\star\in\PP^1(\QQ_p)$. Fix such a pair. The embedding $\Psi$ induces an embedding $\Psi\colon K_p\to B_p$ (where we write $K_p\dfn K\tns \QQ_p$). Recall that we have fixed an isomorphism $\iota\colon B_p\to M_2(\QQ_p)$. The composition $\iota\circ\Psi$ induces an embedding of $K_p^\times/\QQ_p^\times$ into $\PGL_2(\QQ_p)$, which is also noted $\iota\circ\Psi$. This gives an action $*$ of $K_p^\times/\QQ_p^\times$ on the boundary $\PP^1(\QQ_p)$ of $\cH_p$:
\[
\alpha * x\dfn (\iota\circ\Psi)(\alpha)(x),
\]
for $\alpha\in K_p^\times/\QQ_p^\times$ and $x\in \PP^1(\QQ_p)$. Since $p$ is assumed to be inert in $K$, this action is simply-transitive.

The base point $\star\in\PP^1(\QQ_p)$ gives an identification
\[
\eta_{\Psi,\star}\colon K_p^\times/\QQ_p^\times\tto{\cong}\PP^1(\QQ_p),
\]
by sending $1$ to $\star$. The torus $\iota\circ\Psi(K_p^\times)$ has two fixed points in $\cH_p$, which are denoted $z_0$ and $\overline z_0$. They belong to $K_p$ and are interchanged by $\Gal(K_p/\QQ_p)$. In fact, having fixed an embedding of $H(\mu_M)\injects \ol \QQ$ (recall that $H$ is the Hilbert class field of $K$), it is shown in~\cite[Section 5]{MR1614543} how to distinguish $z_0$ from $\ol z_0$. At the cost of an ambiguity in the sign of the subsequent formulas, we can omit this subtlety.

There is a natural homeomorphism $G\cong K_p^\times/\QQ_p^\times$, where $G\dfn K_{p,1}^\times$ denotes the subgroup of $K_p^\times$ of elements of norm $1$. This identification induces an isomorphism
\[
\eta_{\Psi,\star}\colon G\tto{\cong} \PP^1(\QQ_p).
\]
When $\star=\infty$, this is given explicitly by:
\begin{align*}
\eta_{\Psi,\infty}(\alpha)=&\frac{z_0\alpha-\overline z_0}{\alpha-1}, & \eta_{\Psi,\infty}^{-1}(x)=&\frac{x-\overline z_0}{x-z_0}.
\end{align*}

The function $\eta_{\Psi,\star}$ induces in turn a continuous isomorphism:
\[
(\eta_{\Psi,\star})_*=(\eta_{\Psi,\star}^{-1})^*\colon A(G)\to \cA,
\]
from the ring of locally-analytic functions on $G$ to $\cA$.

Consider the polynomial $P_{\Psi(\sqrt{-D_K})}^{\frac{n}{2}}$, where $P_{\Psi(\sqrt{-D_K})}$ has been defined in Equation~\eqref{eq:defpoly} of Section~\ref{section:Cohomology}. Given $\varphi\in A(G)$, we define
\[
\mu_{f,\Psi,\star}(\varphi)\dfn \mu_f\left(P_{\Psi(\sqrt{-D_K})}^{\frac{n}{2}}\cdot(\eta_{\Psi,\star}^{-1})^*(\varphi)\right).
\]

This is a locally analytic distribution on $G$. Now, fix a branch $\log_p$ of the $p$-adic logarithm such that $\log_p(p)=0$. This induces a homomorphism $\log\colon K_p^\times\to K_p$ which vanishes at the roots of unity, thus giving a homomorphism $G\to K_p$. For $s\in\ZZ_p$ and $x\in  G$, define then
\[
x^s\dfn \exp(s\log x).
\]
Let $\Psi$ be an oriented optimal embedding and let $\star\in\PP^1(\QQ_p)$ be a fixed base point.

\begin{df}
The \emphh{partial $p$-adic $L$-function} attached to the datum $(f,(\Psi,\star))$ is the function of the $p$-adic variable $s\in\ZZ_p$ defined by:
\[
L_p(f,\Psi,\star,s)\dfn\int_G x^{s-\frac{n+2}{2}}d\mu_{f,\Psi,\star}(x).
\]
\end{df}

We have so far constructed a distribution on $G$, to which we attach a partial $p$-adic $L$-function. In order to define the anticyclotomic $p$-adic $L$-function we need to consider anticyclotomic extensions of number fields, which we recall now.

\begin{df}
An abelian extension $L/K$ is called \emph{anticyclotomic} if it is Galois over $\QQ$ and if the involution in $\Gal(K/\QQ)$ acts (by conjugation) as $-1$ on $\Gal(L/K)$.
\end{df}

Let $K_\infty$ denote the maximal anticyclotomic extension of $K$ unramified outside $p$. Let $H$ be the Hilbert class field of $K$. Let $K_n$ be the ring class field of $K$ of conductor $p^n$ (so that $K_0=H$). There exists a tower of extensions:
\[
\QQ\subset K\subset H\subset K_1\subset\cdots\subset K_n\subset\cdots.
\]

Assume for simplicity that $\cO_K^\times=\{\pm 1\}$. By class field theory, the $p$-adic group $G$ is isomorphic to $\Gal(K_\infty/H)$. Let $G^n\dfn\Gal(K_\infty/K_n)$, and let $\Delta\dfn\Gal(H/K)$. Write also $\tilde G\dfn\Gal(K_\infty/K)$. These fit into an exact sequence:
\[
1\to G\to\tilde G\to \Delta\to 1,
\]
and in \cite[Lemma 2.13]{bertolini2002tse} it is shown how the natural action of $\Delta\dfn \Pic(\cO)$ on the set of (oriented) optimal embeddings $\emb(\cO,R)$ lifts to an action of $\tilde G$ on the same set. The logarithm $\log_p$ extends uniquely to $\tilde G$, and thus one can define $x^s$ for $s\in\ZZ_p$ and $x\in  G$.

Let $\alpha\in\tilde G$ be an element of $\tilde G$. Given a function $\varphi\colon\tilde G\to \CC_p$, denote by $\varphi_\alpha$ the function $G\to \CC_p$ sending $x$ to $\varphi(\alpha x)$.

For each $\delta\in\Delta$, fix once and for all a lift $\alpha_\delta$ of $\delta$ to $\tilde G$.
\begin{df}
A function $\varphi\colon \tilde G\to \CC_p$ is \emph{locally analytic} if $\varphi_{\alpha_\delta}\in A(G)$ for all $\delta\in\Delta$. The set of locally-analytic functions on $\tilde G$ is denoted by $A(\tilde G)$.
\end{df}

Define a distribution $\mu_{f,K}$ on $A(\tilde G)$ by the formula:
\[
\mu_{f,K}(\varphi)\dfn\sum_{\delta\in\Delta} \mu_{f,\Psi_\delta,\star_\delta}(\varphi_{\alpha_\delta}),
\]
where $(\Psi_\delta,\star_\delta)\dfn \alpha_\delta(\Psi,\star)$.

Finally, define the anticyclotomic $p$-adic $L$-function:
\begin{df}
  The \emphh{anticyclotomic $p$-adic $L$-function} attached to the modular form $f$ and the field $K$ is:
\[
L_p(f,K,s)\dfn\int_{\tilde G} \alpha^{s-\frac{n+2}{2}}d\mu_{f,K}(\alpha),\quad s\in\ZZ_p.
\]
\end{df}

In~\cite[Section~2.5]{bertolini2002tse} it is proven how this function interpolates special values of the classical $L$-function associated to the modular form $f_\infty$ which corresponds to $f$ via Jacquet-Langlands.

\subsection{Values of \texorpdfstring{$L'_p(f,K,s)$}{Lp'(f,K,s)} in terms of Coleman integration on \texorpdfstring{$\cH_p$}{Hp}}
\label{subsection:specialvalues}
Let as before $p$ be an inert prime. By the very definition of $\mu_{f,K}$ as above, the anti-cyclotomic $p$-adic $L$-function $L_p(f,K,s)$ vanishes at the values $s=1,\ldots,n+1$ (see Lemma~\ref{lemma:technical}). One is then interested in the first derivative. Write first
\[
L_p'(f,K,j+1)=\int_{\widetilde G} \log(\alpha)\alpha^{j-\frac{n}{2}}d\mu_{f,K}(\alpha)=\sum_{i=1}^h L_p'(f,\Psi_i,j+1),
\]
where
\[
L_p'(f,\Psi_i,j+1)\dfn \int_G \log(\alpha)\alpha^{j-\frac{n}{2}}d\mu_{f,\Psi_i}(\alpha).
\]
The following formula is a generalization of~\cite[Theorem 3.5]{bertolini2002tse} which, although immediate, is not currently present in the literature:
\begin{thm}
\label{theorem:padicformula}
Let $\Psi$ be an optimal oriented embedding. For all $j$ with $0\leq j\leq n$, the following equality holds:
\[
L_p'(f,\Psi,j+1)= \int_{\ol z_0}^{z_0} f(z)(z-z_0)^{j}(z-\ol z_0)^{n-j}dz,
\]
where the right hand side is to be understood as a Coleman integral on $\cH_p$.
\end{thm}
\begin{proof}
 Start by manipulating the expression for $L_p'(f,\Psi,j+1)$:
\begin{align*}
L_p'(f,\Psi,j+1)&= \int_G\log(\alpha)\alpha^{j-\frac{n}{2}}d\mu_{f,\Psi}(\alpha)\\
&= \int_{\PP_1(\QQ_p)} \log\left(\frac{x-z_0}{x-\ol z_0}\right)\left(\frac{x- z_0}{x-\ol z_0}\right)^{j-\frac{n}{2}}P_{\Psi}^{\frac{n}{2}}(x)d\mu_{f}(x)\\
&=\int_{\PP_1(\QQ_p)}\left(\int_{\ol z_0}^{ z_0}\frac{dz}{z-x}\right)\left(\frac{x- z_0}{x-\ol z_0}\right)^{j-\frac{n}{2}}P_{\Psi}^{\frac{n}{2}}(x)d\mu_{f}(x)
\end{align*}
where the second equality follows from the change of variables $x=\eta_{\Psi}(\alpha)$ and the third from the definition of the logarithm. Note that from the defining property of $\mu_f$ it follows that:
\begin{equation}
\label{equation:tech1}
\int_{\PP_1(\QQ_p)} \frac{\left(\frac{x- z_0}{x-\ol z_0}\right)^{j-\frac{n}{2}}P_{\Psi}^{\frac{n}{2}}(x)}{z-x}d\mu_f(x)=
\int_{\PP_1(\QQ_p)} \frac{\left(\frac{z- z_0}{z-\ol z_0}\right)^{j-\frac{n}{2}}P_{\Psi}^{\frac{n}{2}}(z)}{z-x}d\mu_f(x),
\end{equation}
since the expression:
\[
\frac{\left(\frac{z-z_0}{z-\ol z_0}\right)^{j-\frac{n}{2}}P_{\Psi}^{\frac{n}{2}}(z)-\left(\frac{x- z_0}{x-\ol z_0}\right)^{j-\frac{n}{2}}P_{\Psi}^{\frac{n}{2}}(x)}{z-x}
\]
is a polynomial in $x$ of degree at most $n$. Using Equation~\eqref{equation:tech1}, a change of order of integration and Proposition~\ref{prop:Teitelbaum} we obtain:

\begin{align*}
L_p'(f,\Psi,j+1)&=\int_{\PP_1(\QQ_p)}\int_{\ol z_0}^{z_0}\frac{dz}{z-x}\left(\frac{x- z_0}{x-\ol z_0}\right)^{j-\frac{n}{2}}P_{\Psi}^{\frac{n}{2}}(x)d\mu_f(x)\\
&=\int_{\ol z_0}^{z_0}\left(\int_{\PP_1(\QQ_p)}\frac{d\mu_{f}(x)}{z-x}\right)\left(\frac{z- z_0}{z-\ol z_0}\right)^{j-\frac{n}{2}}P_{\Psi}^{\frac{n}{2}}(z)dz\\
&=\int_{\ol z_0}^{z_0}f(z)\left(\frac{z- z_0}{z-\ol z_0}\right)^{j-\frac{n}{2}}P_{\Psi}^{\frac{n}{2}}(z)dz.
\end{align*}

A justification for the validity of the change of the order of integration can be found in the proof given in~\cite[Theorem 4]{teitelbaum1990voa}.
\end{proof}

\section{A motive}
\label{section:Motive}
In this section we define a certain Chow motive and calculate its realizations. In Subsection~\ref{sec:motives} we explain some notions on Chow motives that will be used in this section and the following. In Subsection~\ref{sec:motiveMn} we recall the motive described in~\cite{iovita2003dpa}. In Subsection~\ref{sec:motiveDn} we modify this definition in the spirit of~\cite{bdp2008}, and define the motive $\cD_n$. The goal of the final two subsections is to compute the realizations of $\cD_n$.

\subsection{Relative motives}
\label{sec:motives}
In this section we introduce the category of relative Chow motives with coefficients in an arbitrary field. We follow the exposition given in~\cite[Section 2]{kunnemann2001chow}.

Let $K$ be a field of characteristic $0$. Let $S$ be a smooth quasiprojective connected scheme over $K$. For simplicity, assume that $S$ is of dimension $1$, as this is the only situation that we will need in the following. Denote by $\Sch(S)$ the category of smooth projective schemes $X\to S$.

\begin{df}
The \emphh{$i$th Chow group} of $X$, written $\chow^i(X)$, is the group of algebraic cycles on $X$ of codimension $i$, modulo rational equivalence.
\end{df}

\begin{df}
The \emphh{Chow ring} of $X$, written $\chow(X)$, is the ring of algebraic cycles on $X$, modulo rational equivalence. The product is given by intersection of cycles.
\end{df}

There is an obvious decomposition, as abelian groups:
\[
\chow(X)=\bigoplus_{i=0}^{d+1} \chow^i(X),
\]
where $d$ is the relative dimension of $X$ over $S$.

\begin{df}
Given $X,Y$ two smooth projective $S$-schemes, the \emphh{ring of $S$-correspondences} is defined as:
\[
\Corr_S(X,Y)\dfn \chow(X\times_S Y).
\]
\end{df}

For $\alpha\in\chow(X_1\times_S X_2)$ and $\beta\in\chow(X_2\times_S X_3)$, the composition of $\alpha$ and $\beta$ is defined as:
\[
\beta\circ \alpha\dfn \pr_{13,*}\left(\pr_{12}^*(\alpha)\cdot\pr_{23}^*(\beta)\right),
\]
where $\pr_{ij}$ is the projection of $X_1\times_SX_2\times_SX_3$ to $X_i\times X_j$.

\begin{df}
A \emphh{projector} on $X$ over $S$ is an idempotent in the ring of relative correspondences $\chow(X\times_S X)$. If $p$ belongs to the $i$th graded piece $\chow^i(X\times_S X)$ we say that $p$ is of degree $i$.
\end{df}

We first introduce the category $\Mot(S)$ of Chow motives over $S$, with respect to ungraded correspondences. Its objects are pairs $(X,p)$, where $X\to S$ is in $\Sch(S)$, and $p$ is a projector. We set:
\[
\Hom_{\Mot(S)}\left((X,p),(Y,q)\right)\dfn q\circ \chow(X\times_S Y)\circ p,
\]
and composition is induced by composition of correspondences. For $i\in\ZZ$, we say that $q\circ \alpha\circ p\in \Hom_{\Mot(S)}\left((X,p),(Y,p)\right)$ is homogeneous of degree $i$ if
\[
q\circ \alpha\circ p\in\bigoplus_\nu\chow^{d_\nu+i}(X_\nu\times_S Y),
\]
where $X=\coprod_\nu X_\nu$ is the decomposition of $X$ into connected components and $d_\nu=\dim(X_\nu/S)$. This makes $\Hom_{\Mot(S)}$ into a graded ring, with multiplication given by composition. Also, given $(X,p)$ and $(Y,q)$ two objects in $\Mot(S)$, we can define $(X,p)\oplus_S(Y,q)\dfn(X\coprod Y,p\coprod q)$ and  $(X,p)\tns[S](Y,q)\dfn(X\times_S Y,p\tns[S]q)$.

\begin{fact}
The category $\Mot(S)$ is an additive, pseudo-abelian $\QQ$-linear tensor-category.
\end{fact}

Next, we define the category $\Mot^0_+(S)$ of effective relative Chow motives. Its objects are those objects $(X,p)$ in $\Mot(S)$ such that $p$ is homogeneous of degree $0$. As morphisms one takes the degree-zero morphisms:
\[
\Hom_{\Mot_+^0(S)}\left((X,p),(Y,q)\right)\dfn\left(\Hom_{\Mot(S)}\left((X,p),(Y,q)\right)\right)^0.
\]

Given an $S$-scheme $X$, one associates to it an object of $\Mot^0_+(S)$:
\[
h(X)\dfn (X,\Delta_X),
\]
where $\Delta_X$ is the diagonal of $X$ in $X\times_S X$. Given a map $f\colon Y\to X$ of $S$-schemes, we can consider (the class of) the transpose of its graph $[\,^t\Gamma_f]$ as an element of $\chow(X\times_S Y)$. Concretely, we consider the map
\[
\gamma_f\colon X\to X\times_S Y,\quad \gamma_f=\id\times f,
\]
and we set $[\Gamma_f]\dfn (\gamma_f)_*[X]$. This can be seen as a morphism
\[
[\Gamma_f]\in \chow^{d}(X\times_S Y)=\Hom_{\Mot^0_+(S)}\left(h(X),h(Y)\right),
\]
and hence $[\phantom{.}^t\Gamma_f]\in\Hom_{\Mot^0_+(S)}(\left(h(Y),h(X)\right)$. Assigning to $f$ the morphism $[\phantom{.}^t\Gamma_f]$ we obtain a contravariant embedding of categories:
\[
h\colon \Sch(S)\to \Mot^0_+(S).
\]

Lastly, define the category $\Mot^0(S)$ of \emphh{relative Chow motives over $S$}. Its objects are triples $(X,p,i)$, where $X$ is a smooth projective $S$-scheme, $p$ is a projector on $X$, and $i$ is an integer. Given $(X,p,i)$ and $(Y,q,j)$ two such objects, we define
\[
\Hom_{\Mot^0(S)}\left((X,p,i),(Y,q,j)\right)\dfn \Hom_{\Mot(S)}^{j-i}\left((X,p),(Y,q)\right).
\]
Composition is again induced from composition of correspondences. In this way, the category $\Mot^0_+(S)$ can be seen as a full subcategory of $\Mot^0(S)$.

\begin{fact}
The category $\Mot^0(S)$ is an additive, pseudo-abelian $\QQ$-category with a canonical tensor product given by:
\[
(X,p,m)\tns(Y,q,n)\dfn =(X\times_S Y,p\tns q,m+n).
\]
\end{fact}

There is also a duality theory: given $M=(X,p,m)$, with $X$ pure of relative dimension $n$ over $S$, define $M^\vee\dfn(X,p^t,n-m)$. Then we have:
\[
\Hom(P\tns M,N)=\Hom(P,M^\vee\tns N).
\]
Define twisting in $\Mot^0(S)$ by
\[
(X,p,m)(n)\dfn(X,p,m+n).
\]
One has a form of Poincar\'e duality: if $d$ is the relative dimension of $X$ over $S$, then:
\[
h(X)^\vee=h(X)(d).
\]
One also has direct sums, which can be described explicitly. Consider first the Lefschetz motive $L_S\dfn(\PP^1_S,\pi_2,0)$, where $\pi_2$ is the K\"unneth projector onto $R^2$, coming from any section to $\PP^1_S\to S$. Since for $m\leq 0$ the motive $(X,p,m)$ is isomorphic to $(X,p,0)\tns[S]L_S^{-m}$, and since the direct sum for motives of degree $0$ is easy:
\[
(X,p,0)\oplus(Y,q,0)\dfn(X\coprod Y,p\coprod q,0),
\]
we can define the direct sum for general objects in $\Mot^0(S)$ as follows: let $r\geq\max(m,n)$. Then $(X,p,m)\oplus_S (Y,q,n)$ is, by definition:
\[
\left(\left(X\times_S(\PP^1_S)^{r-m}\right)\coprod \left(Y\times_S(\PP^1_S)^{r-n}\right),\left(p\tns \pi_2^{\tns(r-m)}\right)\coprod \left(q\tns \pi_2^{\tns(r-n)}\right),r\right).
\]

The importance of Chow motives lies in their universality for the \emphh{realization functors}. For us, this means that given a motive $(X,p,i)$, the correspondence $p$ induces a projector on any Weil cohomology $H^*(X)$, and therefore we obtain functors $H^*$ from the category $\Mot^0(S)$ to the same category where $H^*(X)$ would live, by sending $(X,p,i)$ to $pH^*(X)$. These functors are called \emphh{realization functors}, and we will concentrate in the $l$-adic \'etale and de Rham realizations.

\subsection{The motive \texorpdfstring{$\cM_n^{(M)}$}{Mn-sup-M} of Iovita and Spie\ss{}}
\label{sec:motiveMn}

Fix $M\geq 3$, and let $X_M/\QQ$ be the Shimura
curve parametrizing abelian surfaces with quaternionic multiplication by $\Rmax\subseteq \cB$
and level-$M$ structure, as described in Section~\ref{subsection:shimuracurves}.  Let $\pi\colon \cA\to X_M$ be the
universal abelian surface with quaternionic multiplication. Consider
the relative motive $h(\cA)$ as an object of $\Mot(X_M)$, where $h$ is the contravariant functor
\[
h\colon \Sch(X_M)\to \Mot^0_+(X_M)
\]
from the category of smooth and proper schemes over $X_M$ to the category of
Chow motives, as explained above. In general, the realization functors of a motive give the corresponding cohomology groups, as graded vector spaces with extra structures, and one cannot isolate the $i$th cohomology groups at the motivic level, without assuming what are known as ``standard conjectures''. If the underlying scheme has extra endomorphisms then one can hope to annihilate some of these groups and thus obtain only the desired degree. The following result establishes this for abelian schemes:
\begin{thm}[(Deninger-Murre, K\"unnemann)]
The motive $h(\cA)$ admits a canonical decomposition
\[ h(\cA)=\bigoplus_{i=0}^4 h^i(\cA),
\]
with $h^i(\cA)\cong \wedge^i h^1(\cA)$ and
$h^i(\cA)^{\vee}\cong h^{4-i}(\cA)(2)$.
\end{thm}
\begin{proof}
   This is originally proved in \cite[Theorem 3.1, Proposition 3.3]{deninger1991mda} using the so called ``Fourier theory for abelian schemes''. An explicit closed  formula is given in \cite{kunnemann2001chow}.
\end{proof}

Fix an integer $M\geq 3$. In \cite[Appendix]{iovita2003dpa} the authors
define a motive $\cM^{(M)}_n$ for \emph{even} $n\geq 2$. In this subsection we recall this construction. Let $e_2$ be the unique nonzero
idempotent in $\End(\wedge^2
h^1(\cA))=\End(h^2(\cA))$ such that
\[ x\cdot e_2=\nrd(x)e_2, \text{ for all }x\in \cB.
\]
Define $\epsilon_2$ to be the projector in the ring
$\Corr_{X_M}(\cA,\cA)$ such that
\[
(\cA,\epsilon_2)=\widetilde{\cM}^{(M)}_2\dfn\ker(e_2).
\]
Set $m$ as $n/2$ and define $\widetilde{\cM}^{(M)}_{n}\dfn
\Sym^{m}\widetilde{\cM}^{(M)}_2$. There is a symmetric pairing, given by the cup-product,
\[
h^2(\cA)\tns h^2(\cA)\to \wedge^4 h^1(\cA)\cong \QQ(-2).
\]
Let $\langle\cdot,\cdot\rangle$ be its restriction to $\widetilde{\cM}_2^{(M)}\tns \widetilde{\cM}_2^{(M)}$. It induces a \emphh{Laplace operator}
\[ \Delta_{m}\colon \widetilde{\cM}^{(M)}_{n} \to
\widetilde{\cM}^{(M)}_{n-2}(-2),
\]
given symbolically by
\[
\Delta_m(x_1x_2\cdots x_m)=\sum_{1\leq i<j\leq m} \langle x_i,x_j\rangle x_1\cdots \widehat x_i\cdots \widehat x_j\cdots x_m.
\]

\begin{lemma}[{\cite[Section 10.1]{iovita2003dpa}}]
$\ker(\Delta_m)$ exists (as a motive).
\end{lemma}
\begin{proof}[Sketch of proof]
We will rewrite $\ker(\Delta_m)$ as the kernel of a certain projector, and use the fact that, even if the category of motives is not abelian, at least it has kernels of projectors.

Let $\lambda_{m-2}$ be the morphism
\[ \lambda_{m-2}\colon \widetilde{\cM}_{n-2}^{(M)}\to
\widetilde{\cM}_n^{(M)}(2)
\]
given symbolically by $\lambda_{m-2}(x_1x_2\cdots
x_{m-2})=x_1x_2\cdots x_{m-2} \mu$, where
\[ \mu\colon \QQ\to \widetilde{\cM}_2^{(M)}(2)
\]
is the dual of $\langle\cdot,\cdot\rangle$ twisted by $2$.

Clearly $\Delta_{m}\circ \lambda_{m-2}$ is an isomorphism. Let then
\[ \pr\dfn \lambda_{m-2}\circ (\Delta_{m}\circ
\lambda_{m-2})^{-1}\circ \Delta_m,
\]
which is a projector, so $\ker\Delta_m=\ker(\pr)$ can be written as the kernel of a projector.
\end{proof}

Define the correspondence $\epsilon_n$ in $\Corr_{X_M}(\cA^m,\cA^m)$ to be such that
\[
(\cA^m,\epsilon_n)=(\cM_{n})^{(M)}\dfn \ker(\Delta_m).
\]
\subsection{The motive \texorpdfstring{$\cD_n$}{Dn}}
\label{sec:motiveDn}
Fix $A$ an abelian surface with quaternionic multiplication. Assume
also that $A$ has CM. By Remark~\ref{rmk:aisee}, $A$ is isomorphic to
$E\times E$. Fix such an isomorphism.

Let $\fS_{n}$ be the symmetric group on $n$ letters, and consider the
wreath product $\Xi_{n}\dfn \mu_2\wr \fS_{n}$, which can be described
as the semidirect product
\[ \Xi_{n}\dfn (\mu_2)^{n}\rtimes \fS_{n},
\]
with $\sigma\in\fS_{n}$ acting on $(\mu_2)^{n}$ by
$(x_1,\ldots,x_{n})^\sigma=(x_{\sigma(1)},\ldots,x_{\sigma(n)})$. This is isomorphic to the group of signed permutation matrices of degree $n$.

The group $\Xi_{n}$ acts on $E^{n}$ as follows: each of the
copies of $\mu_2$ acts by multiplication by $-1$ on the corresponding
copy of $E$, and $\fS_{n}$ permutes the $n$ copies.

Let $j\colon \Xi_{n}\to \{\pm 1\}$ be the homomorphism which sends
$-1\in\mu_2$ to $-1$, and which is the sign character on $\fS_{n}$,
and let
\[ \epsilon_E\dfn \frac{1}{2^{n}(n)!}\sum_{\sigma\in\Xi_{n}}
j(\sigma)\sigma\in\QQ[\Aut(E^{n})],
\]
which is an idempotent in the rational group ring of $\Aut(E^{n})$.

By functoriality, $\epsilon_E$ induces a projector in
$\Corr_{X_M}(E^{n},E^{n})$, inducing an endomorphism on the different cohomology groups.

\begin{lemma}[{\cite[Lemma 1.8]{bdp2008}}]\hfill
\begin{enumerate}
\item The image of $\epsilon_E$ action on $\Het^*(E^{n},\QQ_l)$ is
\[
\epsilon_E \Het^*(E^{n},\QQ_l)=\Sym^{n} \Het^1(E,\QQ_l).
\]
\item The image of $\epsilon_E$ acting on $\Hdr^*(E^{n})$ is
\[ \epsilon_E\Hdr^*(E^{n})=\Sym^{n}\Hdr^1(E).
\]
\end{enumerate}
\end{lemma}
\begin{proof}
Denote by $H$ either $\Hdr$ or $\Het$. First note that $-1$ acts as the identity on $H^0(E,\QQ_l)$ and $H^2(E,\QQ_l)$ and as $-1$ on $H^1(E,\QQ_l)$. Therefore all terms in the K\"unneth decomposition
\[
H^*(E^n)=\bigoplus_{(i_1,\ldots,i_n)} H^{i_1}(E)\tns\cdots\tns H^{i_n}(E)
\]
vanish under the action of $\epsilon_E$ except for $H^1(E)^{\tns n}$. The action of $\fS_n$ on this factor is the permutation twisted by the sign character, and thus it induces the projection onto $\Sym^n H^1(E)$.

See also the discussion in~\cite[Lemma~1.8]{bdp2008}.
\end{proof}

 We want to
generalize the construction of~\cite{iovita2003dpa} in the spirit of~\cite{bdp2008}. Let $n$ be a positive even integer, and set $m\dfn n/2$.

\begin{df}
  The motive $\cD_n^{(M)}$ over $X_M$ is defined as:
\[ \cD_n^{(M)}\dfn (\cA^m\times E^n,\epsilon_n^{(M)})\dfn \cM_n^{(M)}\tns (E^{n},\epsilon_E),
\]
where $E^{n}\to X_M$ is seen as a constant family $E^n\times X_M$, with fibers $E^{n}$.
\end{df}

We descend this construction to the Shimura curve $X$. For
that, consider the group $G=(\Rmax/M\Rmax)\cong \GL_2(\ZZ/M\ZZ)$,
which acts canonically (through $X$-automorphisms) on $X_M$, on
$\cA^m$ and on $E^{n}$. Hence we can consider the projector
\[ p_G\dfn \frac{1}{|G|}\sum_{g\in G} g\in \Corr_X(\cA^m\times
E^{n},\cA^m\times E^{n}).
\]

The projector $p_G$ commutes with both $\epsilon_n$ and $\epsilon_E$. In fact, $p_G$ acts trivially on $E^{n}$. So the composition of these projectors is also a projector, which will be denoted $\epsilon$.

\begin{df}
\label{df:gen-kugasato-motive}
The \emphh{generalized Kuga-Sato motive} $\cD_{n}$ is
defined to be
\[ \cD_{n}\dfn  (\cA^m\times E^n,\epsilon)\dfn p_G\left(\cD_n^{(M)}\right)=p_G\left(\cM_{n}^{(M)}\right)\tns
(E^{n},\epsilon_E).
\]
\end{df}

\subsection{The $p$-adic \'etale realization}

Consider the $p$-adic \'etale sheaf $R^2\pi_* \QQ_p$, which has fibers at
each geometric point $\tau\to X_M$ given by $\Het^2(\cA_\tau,\QQ_p)$. We
want to work with a subsheaf of $R^2\pi_* \QQ_p$. For this, note
that the action of $\Rmax$ on $\cA$ induces an action of $\cB^\times$ on
$R^2\pi_*\QQ_p$.

Consider the $p$-adic \'etale sheaf
\[ \bbL_2\dfn \bigcap_{b\in \cB^\times} \ker\left(b-\nrd(b)\colon
R^2\pi_* \QQ_p\to R^2\pi_* \QQ_p \right)\subseteq R^2\pi_*\QQ_p,
\]
which is the subsheaf on which $\cB^\times$ acts as the reduced
norm $\nrd$ of $\cB$. It is a $3$-dimensional locally-free sheaf on
$X_M$. Set $m$ to be $n/2$, and consider the map $\Delta_{m}\colon \Sym^m\bbL_2\to
\left(\Sym^{m-2}\bbL_2\right)(-2)$ given by
the Laplace operator. That is,
\[ \Delta_{m}(x_1\cdots x_m)=\sum_{1\leq i<j\leq m} (x_i,x_j)x_1\cdots
\widehat{x}_i\cdots\widehat{x}_j\cdots x_m,
\]
where $(\cdot,\cdot)$ is the non-degenerated pairing induced from
the cup product and the trace: $(x,y)=\tr(x\cup y)$. Define also
\[ \bbL_n\dfn \ker \Delta_{m},
\]
and
\[
\bbL_{n,n}\dfn \bbL_n\tns \Sym^n \Het^1(E,\QQ_p).
\]
The following lemma gives the $p$-adic \'etale realization of the
motive $\cD_n$.

\begin{lemma}
Consider $\cD_n$ as an absolute motive over $\QQ$. Let $H_p(-)$ be the $p$-adic realization functor. Then:
\[ H_p(\cD_{n})\cong
\Het^1\left(\overline{X_M},\bbL_{n,n}\right)^G=\Het^1\left(\overline{X_M},\bbL_{n}\right)^G\tns \Sym^{n}
\Het^1(E,\QQ_p).
\]
\end{lemma}

\begin{proof} First, note that the $p$-adic realization of the motive
$\cD_{n}^{(M)}$, as thought of as in the derived category, is the complex of $\QQ_p$-sheaves
\[ \bbL_{n}[-n]\tns \Sym^{n} \Het^1(E,\QQ_p).
\]
concentrated in degree $-n$. Then, we just need to compute:
\begin{align*}
H_p(\cD_{n})&=(p_G)_*\left(H^*\left(\overline{X_M},\bbL_n[-n]\tns \Sym^{n}\Het^1(E,\QQ_p)\right)\right)\\
&=\Het^{*-2n}\left(\overline{X_M},\bbL_{n}\right)^G\tns \Sym^{n} \Het^1(E,\QQ_p).
\end{align*}
which follows from the
cohomology of $\bbL_{n}$ being concentrated in degree $1$ and from the
K\"unneth formula.
\end{proof}

\subsection{Semistability and the de Rham realization}
The Hodge filtration on $\Hdr^1(E)\dfn \Hdr^1(E,\QQ_p)$, which is equivalent to the exact sequence
\[
0\to H^0(E,\Omega^1_{E/\QQ_p})\to \Hdr^1(E)\to H^1\left(E,\cO_{E}\right)\to 0
\]
induces a filtration on $\Sym^{n} \Hdr^1(E)$. Write $H_j$ for its $j$th step:
\[
H_j\dfn\Fil^j\left(\Sym^{n}\Hdr^1(E)\right).
\]
The following lemma follows easily from the definitions.
\begin{lemma}
\label{lemma:filderham}
\[
H_j=\begin{cases}
\Sym^{n}\Hdr^1(E) &\If j\leq 0\\
\Sym^{j} H^0(E,\Omega^1_E)\tns\Sym^{2n-j}\Hdr^1(E)&\If 1\leq j\leq n\\
0&\Else .
\end{cases}
\]

\end{lemma}

\begin{thm}[(Faltings, Iovita-Spie\ss{})]
\label{thm:is510}
There is a canonical isomorphism of filtered isocrystals on $\cH_p$:
\[
\pi^*\cHdr^1(\cA/X_M)\cong \cE(M_2).
\]
This isomorphism takes the $\cB_{\QQ_p^\ur}^\times$-action on the left-hand side to the action by $\rho_2$ in the right-hand side.
\end{thm}
\begin{proof}
  See \cite[Lemma~5.10]{iovita2003dpa}.
\end{proof}

Consider the representation $(V_n,\rho_1)$ of $\GL_2$ constructed in Subsection~\ref{subsection:hdrXEV}, and let $\rho_2$ be the one-dimensional representation of $\GL_2$ given by $\det^m$. Then the pair $(V_n,\rho_1,\rho_2)$ induces a filtered convergent $F$-isocrystal
$\cV_n=\cE(V_n\{m\})$ as  described in the first paragraph of Subsection~\ref{subsection:cohomologyXGamma} and in~\cite[Section~4]{iovita2003dpa}. It turns out that it is regular (see~\cite{iovita2003dpa} after Lemma~$4.3$). Moreover, a simple computation using the compatibility of the isomorphism of Theorem~\ref{thm:is510} with tensor products gives the following consequence:
\begin{cor}
\label{cor:is510-0}
\[
\bigcap_{x\in \cB^\times}\ker\bigg((x-\nrd(x))\colon\cE\left(\wedge^2 M_2\right)\to\cE\left(\wedge^2 M_2\right)\bigg)\cong \cV_2.
\]
\end{cor}

We believe that one has a similar result for odd $n$, but we do not formulate a precise statement for it.

There is a map from the space of modular forms on $X_\Gamma$ of weight $n+2$ to $\Fil^{n+1}\Hdr^1(X_\Gamma,\cV_n)$, given by $f(z)\mapsto \omega_f\dfn f(z)\ev_z\tns dz$, where $\ev_z$ is the functional that assigns to a polynomial $R(X)$ its evaluation at the point $z$. Identifying these spaces one obtains the filtration of $\Hdr^1(X_\Gamma,\cV_n)$:
\begin{prop}[{\cite[Proposition 6.1]{iovita2003dpa}}]
\label{prop:is61}
The filtration of $\Hdr^1(X_\Gamma,\cV_n)$ is given by:
\[
\Fil^j\Hdr^1(X_\Gamma,\cV_n)=
\begin{cases}
\Hdr^1(X_\Gamma,\cV_n) &\If j\leq 0,\\
M_k(\Gamma) &\If  1\leq j \leq n+1,\\
0 &\Else .
\end{cases}
\]
\end{prop}

Define the filtered convergent $F$-isocrystal $\cV_{n,n}$ as:
\[
\cV_{n,n}\dfn \cV_n\tns\Sym^n\Hdr^1(E).
\]
Understanding the structure of $\Dst_{\QQ_p^\ur}(H_p(\cD_n))$  will allow us to compute the Abel-Jacobi map in an explicit way. Write $\Hdr^{2n+1}(\cD_n)$ for the filtered $(\phi,N)$-module $\Dst_{,\QQ_p^\ur}(H_p(\cD_n))$. The following key result is a consequence of the facts shown so far.

\begin{thm}
The $G_{\QQ_p}$-representation $H^{2n+1}_p(\cD_n)$ is semistable, and there is a (canonical up to scaling) isomorphism of filtered $(\phi,N)$-modules
\[
\Dst(H_p^{2n+1}(\cD_n))=\Hdr^{2n+1}(\cD_n) \cong \Hdr^1(X_\Gamma,\cV_{n,n})=\Hdr^1(X_\Gamma,\cV_n)\tns[\phantom{\cdot}] \Sym^{n} \Hdr^1(E).
\]
Moreover, writing $\Fil^j$ for $\Fil^j\Hdr^{2n+1}(\cD_{n})$ we have:
\[ \Fil^j=
\begin{cases}
\Hdr^1\left(X_\Gamma,\cV_{n,n}\right) &\If j\leq 0,\\
\Hdr^1\left(X_\Gamma,\cV_n\right)\tns H_j +  M_k(\Gamma)\tns \Sym^{n}\Hdr^1(E) &\If 1\leq j \leq n+1\\
M_k(\Gamma)\tns H_{j-n-1} &\If n+2\leq j \leq 2n+1,\\
0&\Else .
\end{cases}
\]

In particular,
\[
\Fil^{n+1}\Hdr^{2n+1}(\cD_n)\cong M_k(\Gamma)\tns\Sym^{n}\Hdr^1(E).
\]
\end{thm}
\begin{proof}
To prove semistability, we can extend the base to $\QQ_p^\text{ur}$. In this case the curve $X$ is isomorphic to a disjoint union of Mumford curves, and hence it is semistable.

By Corollary~\ref{cor:is510-0}, there is an isomorphism:
\[
\bigcap_{x\in \cB^\times} \ker\left(x-\nrd(x)\colon \cE\left(\wedge^2 M_2\right)\to\cE\left(\wedge^2 M_2\right)\right)\cong \cV_2.
\]

Applying Theorem~\ref{thm:is36} and functoriality, we see that the filtered $(\phi,N)$-module 
\[
\Dst_{,\QQ_p^\text{ur}}\left(\Het^1(\overline{X_M},\bbL_n)\tns \Sym^{n}\Het^1(E,\QQ_p)\right)
\]
is isomorphic to
\[
 \Hdr^1((X_M)_{\QQ_p^\text{ur}},\cV_n)\tns \Sym^{n}\Hdr^1(E/\QQ_p).
\]
This isomorphism can then be descended to $\cD_n$ by taking $G$-invariants.

Putting together Proposition~\ref{prop:is61} with Equation~\eqref{lemma:filderham} we obtain the formula for the filtration.
\end{proof}
\section[Geometric interpretation]{Geometric interpretation of the values of \texorpdfstring{$L'_p(f,K,s)$}{Lp'(f,K,s)}}
\label{section:Final}
This section contains the main result of this project. In Subsection~\ref{subsection:specialvalues} we obtain a formula for the values of the derivative of the $p$-adic $L$-function, in terms of Coleman integrals on the $p$-adic upper-half plane. In Subsection~\ref{sec:cycles} we define a collection of cycles on the motive $\cD_n$ introduced in the previous section, whose image under the $p$-adic Abel-Jacobi map will be calculated. Finally, in Subsection~\ref{sec:maintheorem} we calculate this image and explain the main result.

\subsection{The \texorpdfstring{$p$}{p}-adic Abel-Jacobi map}
\label{sec:AbelJacobi}
Let $X$ be a smooth projective variety over a field $K$. Suppose given a closed immersion $i\colon Z\injects X$  and an open immersion $j\colon U\injects X$, such that $X$ is the disjoint union of $i(Z)$ and $j(U)$. Let $\cF$ be a sheaf on the \'etale site of $X$. Then $i_*i^!\cF$ is the largest subsheaf of $\cF$ which is zero outside $Z$.

The group
\[
\Gamma(X,i_*i^!\cF)=\Gamma(Z,i^!\cF)=\ker\left(\cF(X)\to\cF(U)\right)
\]
is called the group of \emph{sections of $\cF$ with support on $Z$}. The functor which maps a sheaf $\cF$ to $\Gamma(Z,i^!\cF)$
is left-exact, so it makes sense to consider its right-derived functors.

\begin{df}
The functors
\[
H^k_{|Z|}(X,\cF)\dfn \cF\mapsto R^k\Gamma(Z,i^!\cF)
\]
are called the \emphh{\'etale cohomology groups of $\cF$ with support on $Z$}.
\end{df}

Let $\cF$ be a sheaf on the \'etale site of $X$. The short exact sequence of sheaves on the \'etale site of $X$
\[
0\to j_!j^*\cF\to\cF\to i_*i^*\cF\to 0
\]
yields a long exact sequence of abelian groups.
\begin{align*}
0&\to (i^!\cF)(Z)\to F(X)\to F(U)\to\cdots\\
\cdots&\to \Het^k(X,\cF)\to \Het^k(U,\cF)\to H^{k+1}_{|Z|}(X,\cF)\to\cdots
\end{align*}

Let $\ol K$ be the separable closure of $K$, and let $\ol X=X\tns[K]\ol{K}$ be the base change of $X$ to $\ol K$. Assume also that $\ol Z$ is smooth over $\ol K$. Let $c$ be the codimension of $Z$ in $X$. That is, each of the connected components of $\ol Z$ is of codimension $c$ inside the corresponding component of $\ol X$.

Let $\cF$ be a locally constant torsion sheaf on $\ol X$, such that its torsion is coprime to $\car(K)$. In our applications, $\car K=0$, so this condition will be void. As a special case of cohomological purity, (see~\cite{milne1980ec} VI.5.1), we have canonical isomorphisms for every $k\in\ZZ$:
\[
H^k_{|\ol Z|}(\ol X,\cF)\cong \Het^{k-2c}(\ol Z,i^*\cF(-c)).
\]
As a corollary, for $0\leq k\leq 2c-2$ there are $\Het^k(\ol X,\cF)\cong \Het^k(\ol U,\cF)$. Moreover, if $d=\dim(X)$, there is a long exact sequence:
\begin{align}
\label{eqn:gysinmap}
0&\to \Het^{2c-1}(\ol X,\cF)\to \Het^{2c-1}(\ol U,\cF)\to H^{2c}_{|\ol Z|}\left(\ol X,\cF\right)\tto{i_*} \Het^{2c}(\ol X,\cF)\to\cdots\\
 &\cdots\to H^{2d}_{|\ol Z|}\left(\ol X,\cF\right)\to \Het^{2d}(\ol X,\cF)\to \Het^{2d}\left(\ol U,\cF\right)\to 0\nonumber.
\end{align}

We could also replace the group $H^{2c}_{|\ol Z|}\left(\ol X,\cF\right)$ with $\Het^{0}\left(\ol Z,i^*\cF(-c)\right)$, and the group $H^{2d}_{|\ol Z|}\left(\ol X,\cF\right)$ with  $\Het^{2(d-c)}\left(\ol Z,i^*\cF(-c)\right)$.

Assume now that $K$ is a field of characteristic $0$, and let $l$ be a prime. Let $X$ be a smooth projective variety over $K$, and let $\chow^c(X)$ be the Chow group of $X$, consisting of codimension-$c$ cycles with \emph{rational} coefficients. The Chow group has already been introduced in Section~\ref{section:Motive} when discussing a category of relative motives with arbitrary coefficients, but here we work with a simpler setting. Consider the locally-constant sheaves $\cF_n=\ZZ/l^n\ZZ(c)$ in the previous subsection, and take projective limits with respect to $n$, to get $\ZZ_l$-valued cohomology. Inverting $l$ we get $\QQ_l$-valued cohomology, which will be denoted with $\Het$ as well.

The Gysin map $i_*$ in Equation~\ref{eqn:gysinmap} induces by restriction to rational cycles the \emphh{cycle class map} (see \cite[Section~VI.9]{milne1980ec}):
\[
\cl\colon \chow^c(X)\to \Het^{2c}\left(\ol X,\QQ_l(c)\right)^{G_K},
\]
where $\ol X\dfn X\tns[K]\ol K$. Let $\chow_0^c(X)\dfn \ker\cl$. Given a class $[Z]\in \chow^c_0(X)$, represented by a cycle $Z$, consider the short exact sequence of $G_K$-modules:
\begin{equation}
\label{eqn:gysin}
0\to \Het^{2c-1}\left(\ol X,\QQ_l(c)\right)\to \Het^{2c-1}\left(\ol X\setminus |\ol Z|,\QQ_l(c)\right)\to H_{|\ol{Z}|}^{2c}\left(\ol X,\QQ_l(c)\right)_0\to 0,
\end{equation}
where
\[
H_{|\ol{Z}|}^{2c}\left(\ol X,\QQ_l(c)\right)_0\dfn \ker \left( H_{|\ol{Z}|}^{2c}\left(\ol X,\QQ_l(c)\right)\tto{i_*} \Het^{2c}\left(\ol X,\QQ_l(c)\right)\right)
\]
is the kernel of the \emphh{Gysin map}.

Consider the map $\alpha\colon\QQ_l\mapsto H^{2c}_{|Z|}\left(\ol{X},\QQ_l(c)\right)_0$ which sends
\[
1\mapsto \cl_{\ol{Z}}^{\ol{X}}(\ol{Z}) \in H^{2c}_{|\ol Z|}\left(\ol{X},\QQ_l(c)\right).
\]
Set $\ol U\dfn\ol X\setminus |\ol Z|$. Pulling back the exact sequence~\eqref{eqn:gysin} by $\alpha$ we obtain an extension
\begin{equation}
\label{eqn:extaj}
\xymatrix{
0\ar[r]& \Het^{2c-1}\left(\ol X,\QQ_l(c)\right)\ar[r]\ar@{=}[d]& E\ar[r]\ar[d]& \QQ_l\ar[r]\ar[d]^{\alpha}& 0\\
0\ar[r]& \Het^{2c-1}\left(\ol X,\QQ_l(c)\right)\ar[r]&\Het^{2c-1}\left(\ol U,\QQ_l(c)\right)\ar[r]& H_{|\ol{Z}|}^{2c}\left(\ol X,\QQ_l(c)\right)_0\ar[r]& 0.
}
\end{equation}

\begin{df}
The \emphh{$l$-adic \'etale Abel-Jacobi map} is the map
\[
\AJ^\et_l\colon \chow^c_0(X)\to \Ext^1\left(\QQ_l,\Het^{2c-1}\left(\ol X,\QQ_l(c)\right)\right),
\]
which assigns to a class $[Z]$ the class of the extension~\eqref{eqn:extaj}.
\end{df}

Assume now that the variety $X$ is defined
over a $p$-adic field $K$, and set $l=p$. Bloch and Kato in~\cite{bloch1990fat} and Nekovar in~\cite{MR1263527} have defined, for any Galois representation $V$ of $G_K$, a subspace
\[
\Hst^1(K,V)\dfn \ker\left(\Het^1(K,V)\to \Het^1(K,V\tns[\QQ_p] \Bst)\right).
\]
This is identified with the group of extension classes of $V$ by $\QQ_p$ in the category of semistable representations of $G_K$.

The following result can be found in~\cite{MR1738867}.
\begin{lemma}
The image of $\AJ^\et_p$ is contained in
\[
\Hst^1\left(K,\Het^{2c-1}\left(\ol X,\QQ_p(c)\right)\right)\cong \Ext^1_{\Rep_\text{st}(G_K)} \left(\QQ_p,\Het^{2c-1}\left(\ol X,\QQ_p(c)\right)\right).
\]
\end{lemma}

As seen in Section~\ref{section:filteredphiNmodules}, the Fontaine functors $\Dst$ and $\Vst$ give a canonical comparison isomorphism:
\[
\Ext^1_{\Rep_\st(G_K)}\left(\QQ_p,\Het^{2c-1}(\ol X,\QQ_p(c))\right)\cong
\Ext^1_{\MF_K^{\ad,(\phi,N)}}\left(K[c],\Dst(\Het^{2c-1}(\ol X,\QQ_p))\right),
\]
which will make the computations possible.

\subsection{The \texorpdfstring{$p$}{p}-adic Abel-Jacobi for \texorpdfstring{$\cD_n$}{Dn}}
\label{sec:p-adic-abel}
Consider the generalized Kuga-Sato motive $\cD_n=(\cA^m\times E^n,\epsilon)$ as in Definition~\ref{df:gen-kugasato-motive}. The construction of the $p$-adic Abel-Jacobi map of Subsection~\ref{sec:AbelJacobi} can be easily extended to the motive $\cD_n$, by applying the projector at the appropriate places. This can be done for each realization, and we are interested in the  de Rham realization of $\cD_n$, which we have computed to be
\[
\Hdr^1(X_\Gamma,\cV_{n,n}).
\]
This fits in a short exact sequence as in Theorem~\ref{thm:is36}:
\begin{equation}
\label{eq:extXGamma}
0\to \Hdr^1(X_\Gamma,\cV_{n,n})\to\Hdr^1(U,\cV_{n,n})\to \bigoplus_{z\in S} (\cV_{n,n})_z[1],
\end{equation}
where $S$ is a finite set of points in $X_{\Gamma}$ lying in distinct residue classes, and $U$ is the complement in $X_\Gamma$ of $S$. In Section~\ref{section:Final} we will define certain cycles on $\cA^m\times E^n$ which are supported on a fiber above a point $P\in X$. These cycles are of codimension $n+1$, and therefore sending $1$ to their cycle class yields a map:
\[
K[n+1]\to (\cV_{n,n})_z[1].
\]
Pulling back the extension~\eqref{eq:extXGamma} we obtain another extension:
\[
0\to \Hdr^1(X_\Gamma,\cV_{n,n})\to E\to K[n+1]\to 0.
\]
Using Lemma~\ref{lemma:easylemma}, and the fact that the space
\[
\Fil^{n+1}\Hdr^1(X_\Gamma,\cV_{n,n})
\]
is self-orthogonal, we obtain:
\[
\Hdr^1(X_\Gamma,\cV_{n,n})/\Fil^{n+1}\Hdr^1(X_\Gamma,\cV_{n,n})\cong \left(M_{n+2}(\Gamma)\tns\Sym^n\Hdr^1(E)\right)^\vee
\]
The composition map will be denoted $\AJ_K$:
\[
\AJ_K\colon \chow^{n+1}(\cD_n)\to \left( M_{n+2}(\Gamma)\tns\Sym^n\Hdr^1(E)\right)^\vee.
\]

\subsection{Cycles on \texorpdfstring{$\cD_n$}{Dn}}
\label{sec:cycles}
We proceed to define a collection of cycle classes in $\chow^{n+1}(\cD_n)$, indexed by certain isogenies.

Let $E$ be an elliptic curve with complex multiplication by $\cO$. Recall that $\cO=\End_\Rmax(E)$ is an order in an imaginary quadratic
number field $K$. Consider an isogeny $\varphi$ from $E$ to another elliptic curve with complex multiplication $E'$, of degree coprime to $N^+M$. If there is a level-$N^+$ structure and full level-$M$ structure on $E$, we obtain the same structures on $E'$, and also on $A'\dfn E'\times E'$, by putting this level structures only on the first copy. Hence we obtain a point $P_{A'}$ in $X_M$, together with an embedding:
\[
i_{A'}\colon A'\to \cA,
\]
defined over $K$. Let $\Upsilon_\varphi$ be the cycle
\[ \Upsilon_\varphi\dfn (\phantom{.}^t\Gamma_\varphi)^n\subseteq  (E'\times E)^n\cong(A')^m\times E^n\injects \cA^m\times E^n,
\]
where the last inclusion is induced from the canonical embedding
$i_{A'}$, and $\Gamma_\varphi$ is the graph of $\varphi$. Finally, apply the projector $\epsilon$ defined in Section~\ref{section:Motive}.

\begin{df}
  The cycle
\[
\Delta_\varphi\dfn \epsilon \Upsilon_\varphi\in\chow^{n+1}(\cD_n)
\]
is called the \emphh{generalized Heegner cycle} attached to the isogeny $\varphi\colon E\to E'$.
\end{df}

\begin{rmk}
  Since  $H_p(\cD_n)$ is concentrated in degree $2n+1$, the cycle $\Delta_\varphi$ is null-homologous. Therefore it makes sense to study the image of $\Delta_\varphi$ under the any Abel-Jacobi map, in particular the $p$-adic version discussed in Subsection~\ref{sec:AbelJacobi}.
\end{rmk}

We proceed to derive a formula for the image of the Abel-Jacobi map of generalized Heegner cycles, and to compute it in a special setting, thus giving a geometric interpretation of values of the derivative of the anticyclotomic $p$-adic $L$-function.

We have defined a collection $\Delta_\varphi$ of generalized
Heegner cycles of dimension $n$, attached to isogenies $\varphi\colon E\to E'$. Let $\tilde P_{A'}$ be the point of $X_M$ attached to $A'$ through the isogeny $\varphi$. The cycle $\Delta_\varphi$ lies in the $(2n+1)$-dimensional scheme $\cA^m\times E^n$, and so it has codimension $n+1$. As described in Subsection~\ref{sec:p-adic-abel}, we obtain a map:
\[
\AJ_K\colon \chow^{n+1}(\cD_n)\to \left(M_{n+2}(\Gamma)\tns\Sym^n\Hdr^1(E)\right)^\vee,
\]
Let $\omega_f$ be the differential form associated to a modular form $f\in M_{n+2}(\Gamma)$ as explained in Subsection~\ref{sec:modul-forms-shim}. Fix $\alpha\in \Sym^n\Hdr^1(E)$. We want to compute the value:
\[
\AJ_K(\Delta_\varphi)(\omega_f\wedge \alpha)\in \CC_p.
\]

Write $\cl_{P_{A'}}(\Delta_\varphi)$ for the cycle class of $\Delta_\varphi$ on the fiber above $P_{A'}\in X_M$:
\[
\cl_{P_{A'}}(\Delta_\varphi)\dfn \cl_{|\overline{\Delta_\varphi}|}^{\cA^m\times E^n}(\overline{\Delta_\varphi})\in H^{2n+2}_{|\overline{\Delta_\varphi}|}\left(\cA^m\times E^n,\QQ_p(n+1)\right).
\]
Consider the short exact sequence of filtered $(\phi,N)$-modules:
\begin{equation}
\label{eq:extXGamma2}
0\to \Hdr^1(X_{\Gamma},\cV_{n,n})\to \Hdr^1(U,\cV_{n,n})\tto{\res_{P_{A'}}} (\cV_{n,n})_{P_{A'}}[1]\to 0.
\end{equation}

\begin{rmk}
  Here is where we need to exclude the case of weight $2$, which would correspond to $n=0$: in that case $\res_{P_{A'}}$ is always zero, since the restriction map induces an isomorphism
\[
\Hdr^1(X_{\Gamma})\to \Hdr^1(U).
\]
However, in our situation the cokernel of the map $\res_{P_{A'}}$ injects in:
\[
\Hdr^2(X_\Gamma,\cV_{n,n})\cong \Hdr^0(X_\Gamma,\cV_{n,n}^\vee),
\]
where the isomorphism is given by Serre duality. But $\cV_{n,n}$ (and therefore $\cV_{n,n}^\vee$) does not have $\Gamma$-invariants, since it is isomorphic to $n$ copies of the standard representation of $\Gamma$.
\end{rmk}

 We argued that the sequence in Equation~\eqref{eq:extXGamma2} is exact. Its pull-back under the map $1\mapsto\cl_{P_{A'}}(\Delta_\varphi)$ yields then a short exact sequence:
\[
0\to \Hdr^1(X_\Gamma,\cV_{n,n})\to E\to K[n+1]\to 0
\]

Lemma~\ref{lemma:easylemma} ensures that if one forgets the filtration, the resulting sequence of $(\phi,N)$-modules is split, say by a map $s_1\colon K[n+1]\to E$. If we write $([\eta_1],x)$ for $s_1(1)$, note then:
\begin{enumerate}
\item For $([\eta_1],x)$ to be a splitting of the given extension, necessarily $x=1$, and $\eta_1$ has to satisfy:
  \begin{enumerate}
  \item $N_U([\eta_1])=0$, and
  \item $\Phi([\eta_1])=p^{n+1}[\eta_1]$.
  \end{enumerate}
\item For $([\eta_1],1)$ to be in
\[
E=\Hdr^1(U,\cV_{n,n})\times_{(\cV_{n,n})_{P_{A'}}[1]} K[n+1],
\]
necessarily $\res_{P_{A'}}(\eta_1)=\cl_{P_{A'}}(\Delta_\varphi)$.
\end{enumerate}

So let $\eta_1$ be a $\cV_{n,n}$-valued $1$-hypercocycle on $U$ satisfying the conditions:
\begin{align*}
 \res_{P_{A'}}(\eta_1)&=\cl_{P_{A'}}(\Delta_\varphi),& N_U([\eta_1])&=0,& \Phi(\eta_1)=p^{n+1}\eta_1+\nabla G,
\end{align*}
where $G$ is a $\cV_{n,n}$-valued rigid section over $U$. Consider next
\[
[\eta_2]\in \Fil^{n+1}\Hdr^1(U,\cV_{n,n})
\]
such that $\res_{P_{A'}}(\eta_2)=\cl_{P_{A'}}(\Delta_\varphi)$. This element exists as well, because it is the image of $1$ under the splitting $s_2$ of Lemma~\ref{lemma:easylemma}. Let
\[
[\widetilde{\eta}_\varphi]\dfn [\eta_1-\eta_2]\in\Hdr^1(U,\cV_{n,n}).
\]
Then $[\widetilde\eta_\varphi]$ can be extended to all of $X_\Gamma$. That is, there is $[\eta_\varphi]\in \Hdr^1(X_\Gamma,\cV_{n,n})$ such that
\[
j_*([\eta_\varphi])=[\widetilde{\eta}_\varphi]\equiv [\eta_1]\pmod{\Fil^{n+1}\Hdr^1\left(U,\cV_{n,n}\right)}.
\]
Write $[\eta_\varphi]=\iota(c)+t$, with $t\in\Fil^{n+1}\Hdr^1(X_\Gamma,\cV_{n,n})$. Then one can replace $[\eta_2]$ by $[\eta_2]+t$ without changing the properties required for $[\eta_2]$, and hence we can assume that $[\eta_\varphi]=\iota(c)$ for some $c\in \Hdr^1(X_\Gamma,\cV_{n,n})$. Recall the maps $I$ and $P_U$ as appearing in diagram of Equation~\eqref{eq:splittings}. We can prove:

\begin{prop}
With the previous notation, the following equality holds:
\begin{equation}
\label{eq:aj-eq1}
\AJ_K(\Delta_\varphi)([\omega_f]\wedge \alpha)=\left\langle I([\omega_f]\wedge\alpha),P_U([\eta_2])\right\rangle_\Gamma.
\end{equation}
\end{prop}
\begin{proof}
  Using the definition of the Abel-Jacobi map and following the recipe given in Lemma~\ref{lemma:easylemma}, together with the pairings on $\Hdr^1(X,\cV_{n,n})$, we obtain the following equality:
\[
\AJ_K(\Delta_\varphi)([\omega_f]\wedge \alpha)=\left\langle [\omega_f]\wedge\alpha,[\eta_\varphi]\right\rangle_{X_\Gamma}.
\]
The assumption of $\eta_\varphi=\iota(c)$ implies that $I(\eta_\varphi)$ is zero. So the right-hand side can be rewritten, using Equation~\eqref{eqn:cupproduct} and the diagram of Equation~\ref{eq:splittings}, as:
\[
-\left\langle I_{U,c}([\omega_f]\wedge\alpha),P_U([\eta_\varphi])\right\rangle_\Gamma.
\]
Now the result follows from observing that on $U$ one can write $[\eta_\varphi]=[\eta_1]-[\eta_2]$, and that $P_U([\eta_1])=0$.
\end{proof}
The following result computes a formula the right-hand side of Equation~\eqref{eq:aj-eq1}:

\begin{thm}
Let $F_f$ be a Coleman primitive of $\omega_f$, and let $z'_0\in\cH_p$ be a point in the $p$-adic upper-half plane such that $P'_A=\Gamma_M z'_0$. Then:
\[
\left\langle I([\omega_f]\wedge \alpha),P_U([\eta_2])\right\rangle_\Gamma=\left\langle F_f(z'_0)\wedge \alpha,\cl_{z'_0}(\Delta_\varphi)\right\rangle_{V_{n,n}}.
\]
\end{thm}
\begin{proof}
Observe first that the spaces
\[
\Fil^{n+1}\Hdrc^1(U,\cV_{n,n})\text{ and } \Fil^{n+1}\Hdr^1(U,\cV_{n,n})
\]
are orthogonal to each other. Therefore:
\[
0=\langle [\omega_f]\wedge\alpha,[\eta_2]\rangle_U=\left\langle P_{U,c}([\omega_f]\wedge \alpha),I_U([\eta_2])\right\rangle_{\Gamma,U}- \left\langle I([\omega_f]\wedge \alpha),P_U([\eta_2])\right\rangle_\Gamma,
\]
and hence we obtain:
\[
\left\langle I([\omega_f]\wedge \alpha),P_U([\eta_2])\right\rangle_\Gamma=\left\langle P_{U,c}([\omega_f]\wedge \alpha),I_U([\eta_2])\right\rangle_{\Gamma,U}.
\]
In order to show that the right-hand side of the previous equation coincides with
\[
\left\langle F_f(z'_0)\wedge \alpha,\cl_{z'_0}(\Delta_\varphi)\right\rangle_{V_{n,n}},
\]
we use the explicit formula for the pairing as found in Proposition~\ref{prop:formula-for-pairing}. Since $\eta_2$ has only nonzero residue at $P_{A'}$, the right hand side of the formula appearing in Proposition~\ref{prop:formula-for-pairing} reduces to pairing the primitive of $\omega_f\wedge\alpha$ with that residue, at the point corresponding to $P_{A'}$, yielding the desired formula.
\end{proof}

\subsection{Computation and application}
\label{sec:maintheorem}

From here on, let $z_0$ be a point in the $p$-adic upper-half plane $\cH_p$ such that the orbit $P_A\dfn\Gamma z_0$ corresponds to $A=E\times E$ in $X$. Consider the map
\[
g=(\id_E^n,\varphi^n)\colon E^n\to E^n\times (E')^n,
\]
Then $\Delta_\varphi$ is the projection via $\epsilon$ of the image $g(E^n)$. The functoriality of the cycle class map gives:
\[
\left\langle F_f(z'_0)\wedge\alpha,\cl_{z'_0}(\Delta_\varphi)\right\rangle_{V_{n,n}}=\langle \varphi^* F_f(z_0),\alpha\rangle_{V_n},
\]
where now the pairing is the natural one in the stalk $V_n=(\cV_n)_{z_0}$. To compute this last quantity we first do it on a horizontal basis for
\[
\Sym^n\Hdr^1(E/K)=(\cV_n)_{z_0}.
\]
Let $\{u,v\}$ be a horizontal basis for $V_1$, normalized so that $\langle u,u\rangle=\langle v,v\rangle=0$ and $\langle u,v\rangle=-\langle v,u\rangle =1$. This induces a basis $\{v_i\dfn u^iv^{n-i}\}_{0\leq i\leq n}$ of $V_n$.

Choose a global regular section $\omega^n$ in the lowest piece of the filtration which transforms with respect to $\Gamma_M$ as of weight $n$, and scale it so that $\omega^n$ corresponds to $\sum_{i=0}^n (-1)^i\binom{n}{i}z^iv_i$, which is a regular section in $\Fil^n\cV_n$. Note that $\omega^n=(u-zv)^n$, and so:
\[
\langle \omega^n,v_i\rangle=\langle u-zv,u\rangle^i\langle u-zv,v\rangle^{n-i}=z^i.
\]

We want to obtain a formula for the Coleman primitive $F_f$ of $\omega_f$. We proceed by differentiating the section $\langle F_f,v_i\rangle$ and using that $\{v_i\}$ is a horizontal basis:
\[
d\langle F_f,v_i\rangle=f(z)\langle\omega^n,v_i\rangle dz=z^i f(z)dz.
\]
One deduces the formula:
\[
\langle F_f(\varphi(z_0)),\sum_i a_iv_i\rangle=\sum_{i=0}^n a_i\int_{\star}^{\varphi(z_0)} f(z) z^i dz.
\]
From now on we concentrate on the stalk of $\cV_n$ at $z_0$. The chosen regular differential $\omega$ yields a basis element $\omega_{z_0}$ for $\Hdr^1(E,K)$. Choose $\eta_{z_0}$ in the span of $\Phi\omega_{z_0}$ such that
\[
\langle \omega_{z_0},\eta_{z_0}\rangle=1.
\]
This yields a basis for $\Sym^n\Hdr^1(E/K)$, namely $\{\omega_{z_0}^j\eta_{z_0}^{n-j}\}_{0\leq j\leq n}$. We express this basis in terms of the horizontal basis $\{v_i\}$. By construction, we have:
\[
\omega_{z_0}^j\eta_{z_0}^{n-j}=(z_0-\ol z_0)^{j-n}\sum_{i}P_{i,j,n}(z_0,\ol z_0) v_i,
\]
where
\[
P_{i,j,n}(X,Y)=\sum_k\binom{j}{k}\binom{n-j}{k+i-j}(-1)^{n-i}X^kY^{n-i-k}.
\]
The following lemma can be proven by a simple calculation.
\begin{lemma}
  \[
\sum_i P_{i,j,n}(z_0,\ol z_0)z^i=(z-z_0)^j(z-\ol z_0)^{n-j}.
\]
\end{lemma}

The previous lemma gives a formula for a primitive for $\omega_f$:
\begin{align*}
\left\langle \varphi^*F_f(z_0),\omega_{z_0}^j\eta_{z_0}^{n-j}\right\rangle&=(z_0-\ol z_0)^{j-n}\int_{\star}^{\varphi(z_0)} f(z)(z-z_0)^j(z-\ol z_0)^{n-j}dz.
\end{align*}
Note that this equality is only defined up to an ``integration constant'' in $\CC_p$, because in $\cH_p$ the sheaf $\cV_{n,n}$ is trivial. We can now prove the following application:
\begin{thm}
\label{thm:mainthm}
Let $\varphi\colon E\to E'$ be an isogeny of elliptic curves with level-$N$ structure, and let $\ol\varphi$ be the morphism $E\to \ol E'$ obtained by from $\varphi$ by applying to $E'$ the nontrivial automorphism of $K$. Let $\Delta_\varphi^-\dfn \Delta_{\varphi}-\Delta_{\ol\varphi}$, and write $z'_0\in\cH_p$ for the point in the $p$-adic upper-half plane which corresponds to the abelian surface $E'\times E'$. Then there exist a constant $\Omega\in K^\times$ such that, for all $0\leq j\leq n$:
\[
\AJ_K(\Delta_\varphi^-)(\omega_f\wedge\omega^j\eta^{n-j})=\Omega^{j-n}L_p'(f,\Psi_{P_{E'}},j+1).
\]
\end{thm}

\begin{proof}
Set $\Omega$ to be $z_0-\ol z_0$. Since $z_0$ does not belong to the boundary of $\cH_p$, this quantity is non-zero. Using the previous results, we obtain first:
\[
\AJ_K(\Delta_\varphi^-)(\omega_f\wedge \omega^j\eta^{n-j})=\left\langle F_f(z'_0),\omega_{z_0}^j\eta_{z_0}^{n-j}\right\rangle-\left\langle F_f(\ol z'_0),\omega_{z_0}^j\eta_{z_0}^{n-j}\right\rangle,
\]
Therefore, the second term in the previous displayed expression becomes
\[
\left\langle F_f(\ol z'_0),\omega_{z_0}^{n-j}\eta_{z_0}^{j}\right\rangle=\Omega^{j-n}\int_{\star}^{\ol z'_0} f(z)(z- z_0)^j(z-\ol z_0)^{n-j}dz.
\]
Combining this with the formula for $\left\langle F_f(z'_0),\omega_{z_0}^j\eta_{z_0}^{n-j}\right\rangle$ yields:
\[
\AJ_K(\Delta_\varphi^-)(\omega_f\wedge \omega^j\eta^{n-j})=\Omega^{j-n}\int_{\ol z'_0}^{z'_0} f(z)(z- z'_0)^j(z-\ol z'_0)^{n-j}dz.
\]
The result follows now from Theorem~\ref{theorem:padicformula}.
\end{proof}

Note that the integral appearing in the previous theorem coincides with the value at $s=j+1$ of the derivative of the partial $p$-adic $L$-function described before. We obtain the following corollary:
\begin{cor}
\label{cor:mainresult}
Let $H/K$ be the Hilbert class field of $K$, and consider a set of representatives $\{\Psi_1,\ldots,\Psi_h\}$ for $\emb(\cO,\cR)$. For each $\Psi_i$, let $P_i$ be the corresponding Heegner point on $X_H$, and let $\Delta_{\Psi_i}$ be the cycle corresponding to $P_i$. Define $\Delta^-\dfn \sum_i \Delta^-_{\Psi_i}$. There exists a constant $\Omega\in K$ such that for all $0\leq j\leq n$:
\[
\AJ_K(\Delta^-)(\omega_f\wedge\omega^j\eta^{n-j})=\Omega^{j-n}L_p'(f,K,j+1).
\]
\end{cor}
\begin{proof}
  This follows immediately from the expression given in Theorem~\ref{thm:mainthm} for the partial $p$-adic $L$-functions:
\[
\AJ_K(\Delta^-_{\Psi_i})(\omega_f\wedge\omega^j\eta^{n-j})=\Omega^{j-n}L_p'(f,\Psi_i,j+1).
\]
\end{proof}

\begin{rmk}
  There is no canonical choice for the regular differential $\omega\in\Omega^1_{E/K}$. If a given $\omega$ is changed to $\omega_\lambda\dfn\lambda\omega$, with $\lambda\in K$, we  obtain:
\[
\AJ_K(\Delta^-)\left(\omega_f\wedge(\omega_\lambda^j\eta_\lambda^{n-j})\right)=\Omega^{j-n}\lambda^{2j-n}L_p'(f,K,j+1).
\]
Note in particular that the formula at the central point $j=n/2$ does not depend on the choice of the basis of the differential form $\omega$.
\end{rmk}

The formula in Corollary~\ref{cor:mainresult} is the type of result that we were looking for: it relates the values of the derivative of $L_p(f,K,s)$ evaluated at the points $s=1,\ldots, n+1$ to the image under the $p$-adic Abel-Jacobi map of a certain algebraic cycle $\Delta^-$.

\section{Conclusion}
\label{section:Future}
\label{section:Conclusion}
The motive $\cD_n$ should be studied more carefully, as well as the possible families of nontrivial cycles on it. It should be possible to compute the $p$-adic Abel-Jacobi image of more general null-homologous cycles, and this should be compared to their $p$-adic heights.

If one wants to extend the results of this paper to modular forms of odd weight, one needs to work on the two sides of the equation: on the one hand, the motive $\cD_n$ needs to be extended to odd $n$.  On the other hand, the anti-cyclotomic $p$-adic $L$-function as defined in~\cite{bertolini2002tse} does not contemplate possible nebentypes, thus restricting the construction to even-weight modular forms. One should give a more general construction which allowed nebentypes, and these should be incorporated in the definition of the motive $\cD_n$ as well.

It would be interesting to compute the
image of the Abel-Jacobi map for arbitrary cycles on $\cD_{n}$
supported on CM-points of the Shimura curve. Finding explicit formulas for cycles supported at a single point is a more difficult problem than what has been treated in this paper, since some of the techniques used above cannot be used in the general case. However similar computations have been carried over in the split case in~\cite{bdp2008}, and one should be able to adapt them to the setting of this work. A careful choice of these cycles will yield formulae for other
$p$-adic $L$-functions. The underlying philosophy is that all these special
values should come from geometric data, such as algebraic cycles and Abel-Jacobi map.

The focus of this paper has been put on the study of the relation of the anti-cyclotomic $p$-adic $L$-function to the image of certain cycles under the Abel-Jacobi map. More germane to the original Gross-Zagier formulas would be to instead relate the values of the $L$-function to $p$-adic analogues of the N\'eron-Tate heights of the cycles. The investigation of the relation of the $p$-adic Abel-Jacobi map appearing in this paper with the $p$-adic height pairings as in the articles of Gross and Coleman~\cite{MR1097610} and of Nekov\'{a}\v{r} (\cite{MR1263527} and \cite{MR1343644}) would certainly be fruitful. In particular, one should be able to compute the $p$-adic heights of the cycles constructed in this paper, or of generalizations of them, and relate them to values of the anti-cyclotomic $p$-adic $L$-function, or of its derivative.

In \cite{besser1995cco}, the author defines a family of cycles on a variety similar to what we studied in this project, and proves that they span an infinite-dimensional subspace of the Griffiths group of the variety. The author uses the complex Abel-Jacobi map to show the non-triviality of these cycles.  We hope to be able to obtain a similar result by means of the $p$-adic Abel-Jacobi map. Experiments should be carried out in order to collect evidence supporting this nontriviality statement, and this can be easily done using explicit evaluation of the $p$-adic measure attached to a modular form $f$.

\bibliographystyle{amsalpha}
\bibliography{BibMarc}

\end{document}